\documentclass[hidelinks,onefignum,onetabnum]{siamartdong}

\usepackage{epstopdf}
\usepackage{diagbox}
\usepackage{enumerate}
\usepackage{color}

\usepackage{comment}
\usepackage{indentfirst}

\usepackage{subfigure}

\usepackage{multirow}
\usepackage{threeparttable}
\usepackage{amssymb}

\usepackage{lipsum}
\usepackage{amsfonts}
\usepackage{graphicx}
\usepackage{epstopdf}

\usepackage{algorithm}
\usepackage[framemethod=tikz]{mdframed}  
\usepackage{algorithmicx}
\usepackage{array}
\usepackage[ntheorem]{empheq}
\usepackage{algpseudocode}

\ifpdf
  \DeclareGraphicsExtensions{.eps,.pdf,.png,.jpg}
\else
  \DeclareGraphicsExtensions{.eps}
\fi

\newsiamremark{remark}{Remark}
\newsiamremark{hypothesis}{Hypothesis}
\crefname{hypothesis}{Hypothesis}{Hypotheses}
\newsiamthm{claim}{Claim}
\newsiamremark{fact}{Fact}
\crefname{fact}{Fact}{Facts}

\headers{Zeyu Dong, Aqin Xiao, Guojian Yin and Junfeng Yin}{Zeyu Dong, Aqin Xiao, Guojian Yin and Junfeng Yin}

\title{A Relaxed Randomized Averaging Block Extended Bregman-Kaczmarz Method for Combined Optimization Problems\thanks{
\funding{This work was supported by National Natural Science Foundation of China (Grant No. 12471357).}}}

\author{Zeyu Dong\thanks{Shanghai Research Institute for Intelligent Autonomous Systems, Tongji University,
Shanghai 200092, China
  (\email{dongzeyu@tongji.edu.cn}).}
\and Aqin Xiao\thanks{School of Mathematical Sciences, Tongji University, Shanghai 200092, China
  (\email{xiaoaqin@tongji.edu.cn}).}
\and Guojian Yin\thanks{The Institute for Advanced Study, Shenzhen University, Shenzhen, Guangdong 518001, China
  (\email{yin@szu.edu.cn}).}
\and Junfeng Yin\thanks{Key Laboratory of Intelligent Computing and Applications (Ministry of Education), 
School of Mathematical Sciences, Tongji University, 
Shanghai 200092, China
  (\email{yinjf@tongji.edu.cn}).}}
  
\usepackage{amsopn}

\ifpdf
\hypersetup{
  pdftitle={A Relaxed Randomized Averaging Block Extended Bregman-Kaczmarz Method for Combined Optimization Problems},
  pdfauthor={Zeyu Dong, Aqin Xiao, Guojian Yin and Junfeng Yin}
}
\fi

\begin{document}

\maketitle

\begin{abstract}
Randomized Kaczmarz-type methods are widely used for their simplicity and efficiency in solving large-scale linear systems and optimization problems. However, their applicability is limited when dealing with inconsistent systems or incorporating structural information such as sparsity. In this work, we propose a \emph{relaxed randomized averaging block extended Bregman-Kaczmarz} (rRABEBK) method for solving a broad class of combined optimization problems. The proposed method integrates an averaging block strategy with two relaxation parameters to accelerate convergence and enhance numerical stability. We establish a rigorous convergence theory showing that rRABEBK achieves linear convergence in expectation, with explicit constants that quantify the effect of the relaxation mechanism, and a provably faster rate than the classical randomized extended Bregman-Kaczmarz method. Our method can be readily adapted to sparse least-squares problems and extended to both consistent and inconsistent systems without modification. Complementary numerical experiments corroborate the theoretical findings and  demonstrate that rRABEBK significantly outperforms the existing Kaczmarz-type algorithms in terms of both iteration complexity and computational efficiency, highlighting both its practical and theoretical advantages.
\end{abstract}
\noindent{\bf Keywords.} Randomized extended Kaczmarz method,  Bregman-Kaczmarz, Averaging block, Sparse solutions, Least squares problem
 
\section{Introduction}
Consider solving the following combined optimization problem of the form
\begin{equation}\label{eq:main}
 \min_{x \in \mathbb{R}^n} f(x) \quad \text{s.t.} \quad A x = \hat{y}, \quad \text{where} \quad \hat{y} = \arg \min_{y \in \mathcal{R}(A)} g^*(b - y),
\end{equation}
where $A\in \mathbb{R}^{m\times n}$, $x\in \mathbb{R}^{n}$, and $b\in\mathbb{R}^{m}$. 
The objective function $f$ is assumed to be strongly convex but possibly nonsmooth, and $g^*$ denotes a suitable data-misfit function \cite{Schpfer2022ExtendedRK}. In practical applications, available a priori information such as sparsity can be incorporated into the model.  For instance, the choice $f(x) = \lambda \|x\|_1 + \frac{1}{2}\|x\|^2_2$ promotes sparse solutions for appropriately chosen $\lambda > 0$ \cite{chen2001atomic,donoho2006compressed,cai2009convergence}. 
The data-misfit term $g^*(b-y) = \frac{1}{2}\|b - y\|_2^2$ corresponds to a least-squares formulation that penalizes the deviation between the observed data $b$ and the reconstructed vector $y$. 
With these specific choices of $f(x)$ and $g^*(b-y)$, problem \eqref{eq:main} becomes
\begin{equation} \label{eq:main_partical}
    \min_{x \in \mathbb{R}^n} \lambda\|x\|_1 + \frac{1}{2}\|x\|_2^2 
    \quad \text{s.t.} \quad A x = \hat{y}, 
    \quad \text{where} \quad \hat{y} = \arg \min_{y \in \mathcal{R}(A)} \frac{1}{2}\|b - y\|_2^2,
\end{equation}
which leads to sparse least-squares solutions. 
Such optimization problems are widely encountered in applications including compressed sensing \cite{cai2009linearized,yin2008Bregman}, 
signal processing \cite{elad2010sparse,Xiao2025FastBlockNBK}, 
image reconstruction \cite{liang2020deep}, 
and machine learning \cite{adler2017solving,arridge2019solving,benning2018modern}.

The Kaczmarz method \cite{kaczmarz1937} is a classical and extensively studied iterative technique for solving consistent linear systems owing to its algorithmic simplicity and low computational cost per iteration. In the standard (cyclic) Kaczmarz algorithm, each iteration sequentially selects a row of the coefficient matrix and projects the current iterate onto the hyperplane defined by that row equation. To improve the convergence rate, Strohmer and Vershynin proposed the randomized Kaczmarz (RK) method \cite{strohmer2009randomized}, in which the active row is selected randomly with a probability proportional to the squared Euclidean norm of the row.  They established that the RK method achieves linear convergence in expectation for consistent systems, which sparked extensive research on randomized Kaczmarz algorithms. However, the RK method and its extensions are not applicable to inconsistent systems, as random projections do not generally minimize the residual error in that setting.

To remedy this deficiency, Zouzias and Freris introduced the randomized extended Kaczmarz (REK) method \cite{zouzias2013randomized}, which augments the RK framework with an auxiliary projection step that ensures convergence even when the system is inconsistent. The REK method effectively alternates between updating the primary and auxiliary variables, thereby achieving convergence to the least-squares solution. Building upon this framework, a variety of enhanced algorithms have been developed, including block-based formulations \cite{needell2015randomized} and the randomized extended average block Kaczmarz (REABK) method \cite{du2020randomized}, which further improve computational efficiency and robustness. For comprehensive surveys and additional developments, we refer the reader to \cite{2023BW,2019BW,2019D} and the references therein.

More recently, Bregman-type techniques have been incorporated into Kaczmarz iterations to better exploit structural properties of the desired solution, such as sparsity, encoded by suitable convex regularizers.  
The resulting Bregman-Kaczmarz methods \cite{lorenz2014sparse,schopfer2019linear,tondji2023bregman,tondji2023acceleration,tondji2023adaptive} replace the Euclidean metric with a Bregman distance, thereby generalizing the projection rule and enabling the treatment of nonsmooth structure in the objective function $f$.  This class encompasses a broad family of algorithms, including the sparse Kaczmarz method, which has been shown to be effective for computing sparse solutions of regularized Basis Pursuit problems \cite{lorenz2014linearized,lorenz2014sparse}. The classical sparse Kaczmarz method applies cyclic single-row updates and is guaranteed to converge to the unique solution of the underlying inverse problem \cite{lorenz2014linearized}. A randomized sparse block Kaczmarz method, which samples blocks according to prescribed probabilities, was proposed in \cite{petra2015randomized}. By smoothing the objective function, the author established expected linear convergence, while subsequent work \cite{schopfer2019linear} established similar convergence results without smoothing. To further accelerate convergence, an enhanced variant, the randomized block sparse Kaczmarz method with averaging, was later proposed in \cite{tondji2023faster}. 
The Bregman-Kaczmarz framework has also been integrated with various acceleration strategies, including heavy-ball momentum techniques \cite{lorenz2023minimal}, restarting strategies \cite{tondji2023acceleration}, and surrogate-hyperplane constructions \cite{Dong2025RSHKO,Dong2025}. These enhancements broaden the applicability and improve the practical efficiency of Bregman-based Kaczmarz schemes.

To extend the Bregman-based Kaczmarz methods to inconsistent systems, Schöpfer et al.~\cite{Schpfer2022ExtendedRK} proposed the randomized extended Bregman-Kaczmarz (REBK) method for solving \eqref{eq:main}. To the best of our knowledge, REBK is the first randomized Kaczmarz-type algorithm specifically designed for the combined optimization problem. It extends the classical projection framework by incorporating an auxiliary Bregman-update step that ensures convergence in the presence of inconsistency.

Building upon the foundation laid by REBK, we develop a \emph{randomized averaging block extended Bregman-Kaczmarz} (RABEBK) method for solving general combined optimization problems. The RABEBK scheme incorporates an averaging block strategy, allowing multiple rows of the coefficient matrix $A$ to be processed simultaneously, which enhances computational efficiency and reduces the variance inherent in single-row updates. Despite these advantages, the basic RABEBK iteration may still exhibit slow convergence in practice. To alleviate this issue, we introduce a relaxation-parameter mechanism and obtain a relaxed variant, termed rRABEBK, which constitutes the principal contribution of this work. The relaxation parameters provide additional control over the update dynamics and contribute to substantially improved practical performance.

We establish a rigorous convergence theory for the proposed rRABEBK method, showing that it achieves linear convergence in expectation for general combined optimization problems. The analysis provides explicit constants that quantify the influence of the relaxation mechanism on the convergence behavior.  Complementary numerical results further demonstrate that rRABEBK consistently outperforms existing Kaczmarz-type algorithms, including REBK and block-based variants, in both iteration complexity and overall runtime, thereby confirming the practical advantages of the proposed approach.

The remainder of this paper is organized as follows. 
Section \ref{sec:prelim} introduces the necessary preliminaries and notation. 
Section \ref{sec:REBK_review} provides a brief review of the randomized extended Bregman-Kaczmarz (REBK) method. In Section \ref{sec:RABEBK-c}, we propose a randomized averaging block extended Bregman-Kaczmarz (RABEBK) method, together with a relaxation variant (rRABEBK), and establish the corresponding convergence theory. Section \ref{sec:numerical experiments} presents numerical experiments demonstrating the effectiveness and generality of the proposed approaches. Finally, Section \ref{sec:conclusion} concludes the paper with remarks and potential directions for future work.

\section{ Preliminaries}\label{sec:prelim}
In this section, we introduce notation and recall several concepts and properties from convex analysis that will be used throughout the paper.

For a vector $x \in \mathbb{R}^n$, we denote by $x^\top$, $\|x\|_1$, and $\|x\|_2$ its transpose, $\ell_1$-norm, and $\ell_2$-norm, respectively. For a matrix $A \in \mathbb{R}^{m \times n}$, we use $A_{i,:}$ and $A_{:,j}$ to denote its $i$-th row and $j$-th column, and write $A^\top$, $A^\dagger$, $\mathcal{R}(A)$, $\mathcal{N}(A)$, $\|A\|_F$, and $\sigma_{\max}(A)$ for the transpose, Moore-Penrose pseudoinverse, range space, null space, Frobenius norm, and largest singular value of $A$, respectively. 

For index sets $\mathcal{I} \subseteq [m]$ and $\mathcal{J} \subseteq [n]$, we use 
$A_{\mathcal{I},:}$, $A_{:,\mathcal{J}}$, and $A_{\mathcal{I},\mathcal{J}}$
to denote the submatrix of rows indexed by $\mathcal{I}$, columns indexed by $\mathcal{J}$, and rows indexed by $\mathcal{I}$ together with columns indexed by $\mathcal{J}$, respectively.
We call $\{\mathcal{I}_1,\ldots,\mathcal{I}_s\}$ a partition of $[m]$ if 
$\mathcal{I}_i \cap \mathcal{I}_j = \emptyset$ for $i \neq j$ and 
$\bigcup_{i=1}^s \mathcal{I}_i = [m]$.
Similarly, $\{\mathcal{J}_1,\ldots,\mathcal{J}_t\}$ is a partition of $[n]$ if
$\mathcal{J}_i \cap \mathcal{J}_j = \emptyset$ for $i \neq j$ and
$\bigcup_{j=1}^t \mathcal{J}_j = [n]$.
The notation $|\mathcal{I}|$ denotes the cardinality of a set $\mathcal{I} \subseteq [m]$.

The soft-shrinkage operator $S_\lambda(\cdot)$ is defined componentwise as
$$
    (S_{\lambda}(x))_{j} = \max\{ |x_j| - \lambda,\, 0 \} \cdot \mathrm{sign}(x_j),
$$
where $\rm{sign}(\cdot)$ denotes the sign function.

For $\hat{x} \in \mathbb{R}^n$, let
$$
    \mathrm{supp}(\hat{x}) = \{ j \in \{1,\ldots,n\} \mid \hat{x}_j \neq 0 \},
$$
and define
$$
    |\hat{x}|_{\min} = \min \{ |\hat{x}_j| \mid j \in \mathrm{supp}(\hat{x}) \}.
$$

For an index set $J \subseteq \{1,\ldots,n\}$, we denote by $A_J$ the submatrix of $A$ consisting of columns indexed by $J$, and define
$$
    \tilde{\sigma}_{\min}(A)
    = \min \{ \sigma_{\min}(A_J) \mid J \subseteq \{1,\ldots,n\},\, A_J \neq 0 \},
$$
where $\sigma_{\min}(\cdot)$ denotes the smallest nonzero singular value.

Let $f: \mathbb{R}^n\rightarrow \mathbb{R}$ be convex, the subdifferential of $f$ at any $x \in \mathbb{R}^n$ is defined by 
\begin{equation*}
    \partial f ( x ) \stackrel{\text { def }}{=} \{ x ^ { * }  \in \mathbb{R} ^ { n } | f ( y ) \geq f ( x ) + \langle  x ^ { * } , y - x \rangle  , \forall y\in \mathbb{R}^n\}.
\end{equation*}

The Fenchel conjugate of $f$, $f^*:\mathbb{R}^n \rightarrow \mathbb{R}$, is defined as 
$$f^{*}(x^{*})\stackrel{\text { def }}{=}\sup_{y \in \mathbb{R}^{n}} \{\langle x^{*},y \rangle -f(y)\}.$$

If $f$ is differentiable, let $\nabla f(x)$ denote the gradient of $f(x)$, then $\partial f(x) = \{ \nabla f(x)\}$\cite{schopfer2019linear}.

A convex function $f:\mathbb{R}^n \rightarrow \mathbb{R}$ is said to be $\mu$-strongly convex, if for all $x,y \in \mathbb{R}^n$ and subgradients $x^* \in \partial f(x)$, it is satisfied that 
\begin{equation*}
    f(y)\geq f(x)+ \langle x^{*},y-x \rangle + \frac{\mu}{2}\cdot \|y-x\|_{2}^{2}.
\end{equation*}

\begin{lemma} [\cite{rockafellar1998variational}]
If $f:\mathbb{R}^n \rightarrow \mathbb{R}$ is $\mu$-strongly convex then the conjugate function $f^*$ is differentiable with a $1/\mu$ Lipschitz-continuous gradient, i.e.
\begin{equation*}
    \| \nabla f^{*}(x^{*})- \nabla f^{*}(y^{*})\|_{2}\leq \frac{1}{\mu}\cdot \|x^{*}-y^{*}\|_{2},
\end{equation*}
for all $x^{*},y^{*} \in \mathbb{R}^n$.
Consequently,
\begin{equation} \label{eq:alpha-strongly convex}
   f^{*}(y^{*})\leq f^{*}(x^{*})+\langle \nabla f^{*}(x^{*}),y^{*}-x^{*}\rangle + \frac{1}{2\mu}\|y^{*}-x^{*}\|_{2}^{2}. 
\end{equation}
\end{lemma}
\begin{definition}[Bregman distance \cite{bregman1967relaxation}]
    The Bregman distance $D_{f}^{x^{*}}(x,y)$ between $x,y\in \mathbb{R}^n$ with respect to a strongly convex function $f$ and a subgradient $x^* \in \partial f(x)$ is defined as
\begin{equation} \label{eq:Bregman distance}
    D_{f}^{x^{*}}(x,y)\stackrel{\text { def }}{=}{f(y)-f(x)- \langle x^{*},y-x \rangle } =f^{*}(x^{*})- \langle x^{*},y \rangle +f(y).
\end{equation}
\end{definition}

\begin{lemma} [\cite{schopfer2019linear}] \label{lem:Euclid-Bregman}
    Let $f:\mathbb{R}^n \rightarrow \mathbb{R}$ be $\mu$-strongly convex. For all $x,y \in \mathbb{R}^n$ and $x^* \in \partial f(x)$, it holds that
$$ D _ { f } ^ { x^* } ( x , y ) \geq \frac { \mu } { 2 } \| x - y \| _ { 2 } ^ { 2 },$$
and $D _ { f } ^ { x^* } ( x , y )=0$ if and only if $x=y$.
\end{lemma}

\section{The Kaczmarz Method and Some Extensions} \label{sec:REBK_review}
As our work belongs to the class of Kaczmarz-type methods, we begin with a brief introduction to the classical Kaczmarz method and several related variants that are pertinent to our proposed approaches.

The Kaczmarz method \cite{kaczmarz1937} is widely used in scientific and engineering applications due to its simplicity and computational efficiency. The standard Kaczmarz method is designed to solve the consistent linear system $Ax=b$. Given an initial vector $x_0 \in \mathbb{R}^n$, the $(k+1)$-th iterate $x_{k+1}$ is obtained by projecting  $x_k$  onto the hyperplane $$H_{i_k}:\{x:  A^\top _ { i_k,: } x   = b _ { i_k }\},$$ where the row index $i_k$ is selected cyclically from $[m] = \{1,2,\ldots,m\}$.  Precisely, the update rule can be written explicitly as
\begin{equation} \label{eq:classicalKaczmarz}
        x _ { k + 1 }  = x _ { k } + \frac {b _ { i_k } - A _ { i_k,: } x _ { k } } { \| A _ { i_k,:  }\| _ { 2 } ^ { 2 } } \cdot (A _ { i_k,:  })^\top.
\end{equation}

However, the standard Kaczmarz method is not applicable to inconsistent systems. To address this limitation, several extended Kaczmarz methods have been developed \cite{du2020randomized,needell2015randomized,zouzias2013randomized}. A prominent example is the randomized extended Kaczmarz (REK) method introduced in \cite{zouzias2013randomized}. Recall that any right-hand side $b$ of an inconsistent system $Ax=b$ admits the orthogonal decomposition $$b = \hat{y}+z,\qquad \hat{y} \in \mathcal{R}(A),\  z\in  \mathcal{N}(A^\top),$$ where $\hat{y}$ is the projection of $b$ onto the range space of $A$, and $z$ lies in the null space of $A^\top$. Consequently, solving the inconsistent system $Ax=b$ can be reformulated in terms of the following pair of consistent systems:
\begin{equation} \label{eq:BK-2}
Ax = \hat{y},\qquad A^{\top}z = 0.
\end{equation}
The key idea behind the REK method is alternating between randomized Kaczmarz updates for the two consistent systems in \eqref{eq:BK-2} at each iteration, thereby ensuring convergence even when the original system is inconsistent. Applying the randomized Kaczmarz update to each consistent system in \eqref{eq:BK-2} leads to the following REK iteration:
\begin{equation} \label{eq:extendedclassicalKaczmarz}
    \begin{split}
        z_{k+1}&=  z_{k}-\frac{(  A_{:,j_k})^\top   z_{k}}{\|A_{:,j_k}\|_2^2}  A_{:,j_k},\\
   x_{k+1}&= x_{k}-\frac{  A_{i_k,:}  x_{k}-  b_{i_k}+  z_{k+1,i_k}}{  \|A_{i_k,:}\|_2^2 }(  A_{i_k,:})^\top,
    \end{split}
\end{equation}
where the column index $j_k$ and the row index $i_k$ are selected with probabilities
\[
\Pr(\text{index} = j_k) = \frac{\|A_{:,j_k}\|_2^2}{\|A\|^2_F}, \quad
\Pr(\text{index} = i_k) = \frac{\|A_{i_k,:}\|_2^2}{\|A\|^2_F}.
\]
 
Recently, Lorenz et al. \cite{lorenz2014sparse} introduced a Bregman-type variant of the Kaczmarz method, termed the \emph{Bregman--Kaczmarz} (BK) method, for solving the constrained optimization problem\begin{equation} \label{eq:BK}
    \min_{x \in \mathbb{R}^n} f(x) \quad \mathrm{s.t.}\quad Ax=b,
\end{equation}
where $f$ is a strongly convex, possibly nonsmooth objective function.
The BK algorithm proceeds as follows. Given initialization $x_0^{*}=0$ and define $x_0=\nabla f^{*}(x_0^*)$, the approximate solution $x _ { k + 1 }$ at the $k$-th iteration is computed  according to the update scheme
\begin{equation} \label{eq:Bregman-Kaczmarz2}
    \begin{split}
        &x _ { k + 1 } ^ { * } = x _ { k } ^ { * } +  \frac {b _ { i_k } - A _ { i_k,: } x _ { k } } { \| A _ { i_k,:  }\| _ { 2 } ^ { 2 } } \cdot (A _ { i_k,:  })^\top, \\
        &x _ { k + 1 } =\nabla f^{*}(x^{*}_{k+1}) .
    \end{split}
\end{equation}
The update scheme \eqref{eq:Bregman-Kaczmarz2} consists of two stages \cite{tondji2023bregman}. The first stage,  referred to as the \emph{dual space update}, resembles the Kaczmarz update \eqref{eq:classicalKaczmarz}. The second stage, known as the \emph{primal space update}, maps the dual iterate $x_{k+1}^*$ to the primal iterate $x_{k+1}$ through the operator $\nabla f^*$, yielding the $(k+1)$-th approximate solution.

The BK method encompasses several well-known Kaczmarz-type algorithms as special cases, depending on the choice of the objective function $f$. For instance, if $f(x) = \frac{1}{2}\|x\|_2^2$,  then $f^*=f$ and $\nabla f^*(x) = x$. In this case, the BK method reduces to the standard Kaczmarz method \eqref{eq:classicalKaczmarz}, which yields the minimum-norm solution. Alternatively, selecting $f(x) = \lambda \|x\|_1 + \frac{1}{2}\|x\|^2_2$ in \eqref{eq:BK} promotes sparsity in the recovered solution. Under this choice, the mapping $x_{k+1}=\nabla f^*(x^*_{k+1})$ becomes a componentwise soft-thresholding operation,
 $$x _ { k + 1 } =S_{\lambda}(x^{*}_{k+1}),$$ 
where $S_{\lambda}(\cdot)$ denotes the soft-thresholding operator.
In this case, the BK method reduces to the sparse Kaczmarz method \cite{lorenz2014sparse}.

To make the Bregman-Kaczmarz method applicable for the inconsistent systems, Schöpfer et al.~\cite{Schpfer2022ExtendedRK} proposed the randomized extended Bregman-Kaczmarz (REBK) method for solving \eqref{eq:main}. The REBK algorithm can be regarded as a combination of the REK and BK strategies. Therefore,  it is easy to arrive at the update scheme of REBK, that is,
\begin{equation} \label{eq:REBK_scheme}
    \begin{split}
   z_{k+1}^*&=  z_{k}^*-\frac{(  A_{:,j_k})^\top   z_{k}}{\|A_{:,j_k}\|_2^2}  A_{:,j_k},\\
   z_{k+1}&= \nabla g^*(z_{k+1}^*);\\
   x_{k+1}^*&= x_{k}^*-\frac{  A_{i_k,:}  x_{k}-  b_{i_k}+  z^*_{k+1,i_k}}{  \|A_{i_k,:}\|_2^2 }(  A_{i_k,:})^\top, \\
     x_{k+1}&= \nabla f^*(x_{k+1}^*).
    \end{split}
\end{equation}

In particular, when the data-misfit function and objective function are chosen as  \( g^*(b - y) = \frac{1}{2} \| b - y \|_2^2 \) and  $f(x) = \frac{1}{2}\|x\|_2^2$, respectively,  the REBK method reduces to the randomized extended Kaczmarz method \eqref{eq:extendedclassicalKaczmarz}, which converges to the minimum-norm least-squares solution.

\section{The Proposed Methods} \label{sec:RABEBK-c}
The objective of this work is to improve the numerical performance of the REBK method for solving the combined optimization problem \eqref{eq:main}. We pursue this goal in two stages. First, we introduce an averaged block extension of the REBK iteration \eqref{eq:REBK_scheme}. This block formulation allows multiple row projections to be carried out in parallel and subsequently averaged,
thereby enhancing the stability of the iterative process. The resulting scheme forms the foundation of our proposed framework.

Building on this averaged block structure, we further incorporate a relaxation parameter into the update rule. The relaxation step provides an additional degree of freedom for balancing aggressive corrections with numerical robustness, enabling the method to better accommodate inconsistent systems and regularized formulations. With an appropriately chosen relaxation factor, the resulting algorithm exhibits improved convergence behaviour and enhanced practical performance. 

These developments lead to the randomized averaging block extended Bregman-Kaczmarz (RABEBK) method and its relaxed variant, rRABEBK, which will be described in detail in the subsequent subsections.

\subsection{A Randomized Averaging Block Extended Bregman-Kaczmarz Method} \label{sec:RABEBK-d}
It is straightforward to verify that the dual problem associated with
\begin{equation} \label{eq:SLTW17}
    \min_{y \in \mathcal{R}(A)} g^*(b - y),
\end{equation}
 is the constrained optimization problem
\begin{equation} \label{eq:BK01}
    \min_{z \in \mathbb{R}^m} h(z): = g(z)-b^\top z \quad \mathrm{s.t.}\quad A^\top z=0,
\end{equation}
under the assumption that \( g \) is strongly convex, which can guarantee zero duality gap between the primal-dual pair \eqref{eq:SLTW17}-\eqref{eq:BK01}. If we would like to use the BK method \eqref{eq:Bregman-Kaczmarz2} to solve the dual problem \eqref{eq:BK01}, then the specific update scheme is of the form
\begin{equation} \label{eq:BK03}
    \begin{split}
        &\tilde{z}_{k+1}^* = \tilde{z}_k^* -\frac{(  A_{:,j_k})^\top   z_{k}}{\|A_{:,j_k}\|_2^2}  A_{:,j_k}, \\
        &z_{k+1} = \nabla h^*(\tilde{z}_{k+1}^*), 
    \end{split}
 \end{equation}
with initialization \( \tilde{z}_0^* = 0 \). Note that  \( \nabla h^*(\tilde{z}_k^*) = \nabla g^*(\tilde{z}_k^* + b) \), it is convenient to introduce the change of variables  \( z_k^* := \tilde{z}_k^* + b \). Substituting this identity into \eqref{eq:BK03}, we obtain the equivalent form
\begin{equation} \label{eq:BK04}
    \begin{split}
        &z_{k+1}^* = z_{k}^*-\frac{(  A_{:,j_k})^\top   z_{k}}{\|A_{:,j_k}\|_2^2}  A_{:,j_k}, \\
        &z_{k+1} = \nabla g^*(z_{k+1}^*), 
    \end{split}
 \end{equation}
with initialization \( z_0^* = b \). This formulation represents the dual update in the extended Bregman-Kaczmarz framework and will be used later to establish its connection with the REBK method.

In view of the REBK update scheme \eqref{eq:REBK_scheme} for solving the combined optimization problem \eqref{eq:main}, the REBK method can therefore be interpreted as applying the BK method \eqref{eq:Bregman-Kaczmarz2}, together with randomized selection, alternately to the two constrained optimization problems \eqref{eq:BK01} and 
\begin{equation} \label{eq:BK02}
    \min_{ x\in \mathbb{R}^n} f(x) \quad \mathrm{s.t.}\quad Ax=\hat{y},
\end{equation}
with the relation $b = \hat{y}+\nabla g(z)$ derived from first-order optimality condition. 

More precisely, the variable \(z\) is iteratively updated so as to enforce the constraint \(A^\top z = 0\), thereby extracting the component of the data \(b\) lying in the null space \(\mathcal{N}(A^\top)\). 
Alternately, the variable \(x\) is updated to satisfy the linear constraint \(Ax = \hat{y} \in \mathcal{R}(A)\), where \(\hat{y}\) represents the component of \(b\) lying in the range space \(\mathcal{R}(A)\). 
Consequently, the REBK iteration proceeds by alternating between correcting the residual in the auxiliary constraint \(A^\top z = 0\) and refining the primary variable \(x\) with respect to the feasible affine subspace \(Ax = \hat{y}\). Taken together, these complementary updates jointly reconstruct the solution to \eqref{eq:main}.

Thus, at the $k$-th iteration, the REBK method needs to updates two approximate solutions,
$z_{k+1}$ and $x_{k+1}$, using information from the hyperplanes $H_{j_k}:\{z:  A^\top _ { :, j_k} z  = 0\}$ and  $\hat{H}_{i_k}:\{x:  A _ { i_k,: } x   = \hat{y}_ { i_k }\}$, respectively. Note that both $H_{j_k}$ and  $\hat{H}_{i_k}$ involve single-row information, it is natural to consider leveraging averaging block techniques \cite{du2020randomized,necoara2019faster}, which exploit information from multiple hyperplanes simultaneously and average their corresponding projection directions, typically leading to improved stability and faster convergence. Motivated by this idea, we introduce the \emph{randomized averaging block extended Bregman-Kaczmarz} method, abbreviated as RABEBK.

 The update scheme of the RABEBK method can be described as follows. Given initialization $z_0^{*}=b$, $z_0=\nabla g^*(z_0^*)$, $x_0^{*}=0$ and $x_0=\nabla f^{*}(x_0^*)$. At the $k$-th iteration, the method performs the update
\begin{equation} \label{eq:RABEBK_scheme}
    \begin{aligned}
   z_{k+1}^*&=  z_{k}^*-\sum_{l\in\mathcal{J}_{j_k}}\omega_{l}\frac{(  A_{:,l})^\top   z_{k}}{\|A_{:,l}\|_2^2}  A_{:,l},\qquad \omega_l=\frac{ \|A_{:,l}\|_2^2}{\|  A_{:,\mathcal{J}_{j_k}}\|_F^2},\\
   z_{k+1}&= \nabla g^*(z_{k+1}^*);\\
   x_{k+1}^*&= x_{k}^*-\sum_{q\in\mathcal{I}_{i_k}}\omega_q\frac{  A_{q,:}  x_k-  b_q+  z^*_{k+1,q}}{  \|A_{q,:}\|_2^2 }(  A_{q,:})^\top, \qquad \omega_q=\frac{  \|A_{q,:}\|_2^2}{\|  A_{\mathcal{I}_{i_k},:}\|_F^2},\\
     x_{k+1}&= \nabla f^*(x_{k+1}^*).
    \end{aligned}
\end{equation}
At each iteration,  a column block \(\mathcal{J}_{j_k}\) and a row block \(\mathcal{I}_{i_k}\) are selected at random, with probabilities proportional to the Frobenius norms of the associated submatrices:
\[
\Pr(\text{index} = j_k) = \frac{\|A_{:,\mathcal{J}_{j_k}}\|_F^2}{\|A\|_F^2}, \quad
\Pr(\text{index} = i_k) = \frac{\|A_{\mathcal{I}_{i_k},:}\|_F^2}{\|A\|_F^2}.
\]
This sampling strategy favours blocks that have larger energy, as measured by their Frobenius norm, thereby increasing the likelihood of selecting the most informative parts of the matrix.
Once a block is selected, weighted contributions are computed from its constituent rows and columns. Specifically, for each  \(l \in \mathcal{J}_{j_k}\), and  \(q \in \mathcal{I}_{i_k}\), we assign the weights
\[
\omega_l = \frac{\|A_{:,l}\|_2^2}{\|A_{:,\mathcal{J}_{j_k}}\|_F^2}, \quad \omega_q = \frac{\|A_{q,:}\|_2^2}{\|A_{\mathcal{I}_{i_k},:}\|_F^2},
\]
which serve to amplify the influence of those rows and columns with larger Euclidean norms, thereby reflecting their relative importance in the update. In this way, directions with larger weights naturally play a stronger role in the update. The weighted projections are then aggregated to update the dual iterates \(z_{k+1}^*\) and \(x_{k+1}^*\), which are subsequently mapped to the primal variables \(z_{k+1}\) and \(x_{k+1}\) using \(\nabla g^*\) and \(\nabla f^*\), respectively. By simultaneously utilizing multiple hyperplanes in each iteration, RABEBK is able to incorporate richer geometric information than its single-hyperplane counterpart, thereby improving its numerical efficiency.

It is worth noting that the proposed RABEBK method reduces to the REBK method \eqref{eq:REBK_scheme} as a special case when the block size is one, that is, when each selected block contains a single row or column. If we want to find the minimum-norm solution to a least-squares problem and consider the objective function $f(x) = \frac{1}{2}\|x\|_2^2$ along with misfit function $g^*(z) = \frac{1}{2}\|z\|_2^2$, the  RABEBK method specializes  to the randomized extended average block Kaczmarz method proposed in \cite{du2020randomized}.

\subsection{A Randomized Averaging Block Extended Bregman-Kaczmarz Method with Relaxation Parameters} \label{sec:RABEBK-e}
The RABEBK method employs an averaging block strategy to enhance the numerical performance of the REBK scheme; however, this is not the whole story of the present work.
To further improve convergence behaviour, we attempt to introduce relaxation parameters to the update scheme \eqref{eq:RABEBK_scheme}  of the RABEBK method, which leads to the following iteration:
\begin{equation}\label{eq:rRABEBK_scheme}
    \begin{aligned}
   z_{k+1}^*&=  z_{k}^*-\alpha^{(z)}\left(\sum_{l\in\mathcal{J}_{j_k}}\omega_{l}\frac{(  A_{:,l})^\top   z_{k}}{\|A_{:,l}\|_2^2}  A_{:,l}\right),\qquad \omega_l=\frac{ \|A_{:,l}\|_2^2}{\|  A_{:,\mathcal{J}_{j_k}}\|_F^2},\\
   z_{k+1}&= \nabla g^*(z_{k+1}^*);\\
   x_{k+1}^*&= x_{k}^*-\alpha^{(x)}\left(\sum_{q\in\mathcal{I}_{i_k}}\omega_q\frac{  A_{q,:}  x_k-  b_q+  z^*_{k+1,q}}{  \|A_{q,:}\|_2^2 }(  A_{q,:})^\top \right), \qquad \omega_q=\frac{  \|A_{q,:}\|_2^2}{\|  A_{\mathcal{I}_{i_k},:}\|_F^2},\\
     x_{k+1}&= \nabla f^*(x_{k+1}^*).
    \end{aligned}
\end{equation}
Here \(\alpha^{(z)}\) and \(\alpha^{(x)}\) denote the relaxation parameters. A key question is how to choose appropriate values for these parameters. In our method, we set 
$$ \alpha^{(z)} = \frac{\mu_g}{\beta^{\mathcal{J}}_{\max}}\quad\text{and}\quad   \alpha^{(x)} = \frac{\mu_f}{\beta^{\mathcal{I}}_{\max}}, $$ 
where  
\begin{equation}\label{eq:25-23-2}
\beta^{\mathcal{I}}_{\max} := \max_{i_k \in [s]} \frac{\sigma^2_{\max} \left( A_{\mathcal{I}_{i_k},:} \right)}{\| A_{\mathcal{I}_{i_k},:} \|^2_F} \quad\text{and}\quad \beta^{\mathcal{J}}_{\max} := \max_{j_k \in [t]} \frac{\sigma^2_{\max} \left( A_{:,\mathcal{J}_{j_k}} \right)}{\| A_{:,\mathcal{J}_{j_k}} \|^2_F}.
\end{equation}

The rationale behind these choices will be detailed in Remarks \ref{remark:zk_general} and \ref{remark:xk_general}. Because scheme \eqref{eq:rRABEBK_scheme} augments RABEBK with two fixed relaxation parameters that adjust the step sizes of the auxiliary and primary updates, we term the resulting method the \emph{relaxed randomized averaging block extended Bregman-Kaczmarz} (rRABEBK) algorithm. Its implementation is summarized in Algorithm \ref{alg:cRABEBK}.
\begin{algorithm}[!htbp]
 \caption{A Relaxed Randomized Averaging Block Extended Bregman-Kaczmarz Method (rRABEBK)} \label{alg:cRABEBK}
 	\begin{algorithmic}[1]
  \Require matrix $A\in \mathbb{R}^{m\times n}$, $b\in \mathbb{R}^m$ and the maximum iteration $T$. Set $x^{*}_{0}=0$, $x _ { 0} = \nabla f^{*}(x^{*}_{0})$, $z_0^{\ast}=b$ and $z_0=\nabla g^*(z_0^*)$. Set the relaxation parameters  \( \alpha^{(z)} = \frac{\mu_g}{\beta^{\mathcal{J}}_{\max}} \) and \( \alpha^{(x)} = \frac{\mu_f}{\beta^{\mathcal{I}}_{\max}} \).
    \Ensure (approximate) solution of \par $\min_{x \in \mathbb{R}^n} f(x) \quad \text{s.t.} \quad A x = \hat{y}, \quad \text{where} \quad \hat{y} = \arg \min_{y \in \mathcal{R}(A)} g^*(b - y).$
  \State let $\left\{\mathcal{I}_1,\mathcal{I}_2,\cdots,\mathcal{I}_s \right\}$ and $\left\{\mathcal{J}_1,\mathcal{J}_2,\cdots,\mathcal{J}_t \right\}$
		be partitions of $[m]$ and $[n]$, respectively.
		 \For {$k=0, 1,\ldots,T $ }  
   \State select $j_k\in [t]$ with probability ${\rm Pr(index}= j_k)= \frac{\|A_{:,\mathcal{J}_{j_k}}\|_F^2}{\|A\|_F^2}$ 
		\State set $z_{k+1}^*= z_{k}^* - \frac{\alpha^{(z)}} {\lVert A_{:,\mathcal{J}_{j_k}}\rVert^2_F} A_{:,\mathcal{J}_{j_k}}(A_{:,\mathcal{J}_{j_k}})^\top z_{k}$
  \State  update $z_{k+1}=\nabla g^*(z_{k+1}^*)$
	   \State select $i_k\in [s]$ with probability ${\rm Pr(index}= i_k)= \frac{\|A_{\mathcal{I}_{i_k},:}\|_F^2}{\|A\|_F^2}$ 
 \State set $ x_{k+1}^{\ast}=x_{k}^{\ast}-\frac{\alpha^{(x)}}{\|A_{\mathcal{I}_{i_k},:}\|_F^2} (A_{\mathcal{I}_{i_k},:})^\top(A_{\mathcal{I}_{i_k},:}x_k-b_{\mathcal{I}_{i_k}}+z^*_{k+1,\mathcal{I}_{i_k}})$
    \State update  $ x_{k+1}=\nabla f^{\ast}( x_{k+1}^{\ast} )$.
    \State if a stopping criterion is satisfied, stop and \text{return} $x_{k+1}$.
 		\EndFor 
 	\end{algorithmic}
 \end{algorithm}
 
 In the remainder of this section, we focus on the convergence analysis of the rRABEBK method.
Before presenting the main convergence result, we briefly recall the definitions of calmness, linearly regular and grow at most linearly \cite{Schpfer2022ExtendedRK} for completeness. It is worth noting that these properties are satisfied when the function $f$ has a Lipschitz-continuous gradient or when $f$ is a convex piecewise linear-quadratic function, such as $f(x)=\lambda\|x\|_1+\frac{1}{2}\|x\|_2^2$.

\begin{definition} \label{def:2.2}
The (set-valued) subdifferential mapping $\partial f : \mathbb{R}^n \rightrightarrows \mathbb{R}^n$ is said to be \emph{calm} at $\hat{x}$ if there exist constants $c > 0$ and $L > 0$ such that
\begin{equation*}
    \partial f(x) \subset \partial f(\hat{x}) + L \|x - \hat{x}\|_{2} \, B_{2}
    \quad \text{for any } x \text{ with } \|x - \hat{x}\|_{2} \le c,
\end{equation*}
where $B_{2}$ denotes the closed unit ball in the $\ell_{2}$ norm.
\end{definition}

\begin{definition} \label{def:2.3}
 Suppose that $\partial f(x)\cap \mathcal{R}(A^{\top})\neq \emptyset$. The collection $\{\partial f(\hat{x}),\, \mathcal{R}(A^{\top})\}$ is said to be linearly regular if there exists a constant $\zeta > 0$ such that, for all $x^* \in \mathbb{R}^n$, 
\begin{equation*} \label{eq:linearly regular}
    \mathrm{dist}(x^{*}, \partial f(\hat{x})\cap \mathcal{R}(A^{\top}))\leq \zeta \cdot\Big(\mathrm{dist}(x^{*}, \partial f(\hat{x}))+\mathrm{dist}(x^{*},\mathcal{R}(A^{\top}))\Big).
\end{equation*}
\end{definition}

\begin{definition} \label{def:2.4}
    The subdifferential mapping $\partial f$ is said to grow at most linearly if there exist constants $\rho_{1}, \rho_{2} \ge 0$ such that, for all $x \in \mathbb{R}^{n}$ and all $x^* \in \partial f(x)$,
    \begin{equation*}
        \|x^*\|_{2} \le \rho_{1} \, \|x\|_{2} + \rho_{2}.
    \end{equation*}
\end{definition}

\begin{theorem}[\cite{Schpfer2022ExtendedRK}, Theorem 3.9]  \label{thm:theta}
    Consider the problem \eqref{eq:BK} with a strongly convex function $f:\mathbb{R}^n \rightarrow \mathbb{R}$. If its subdifferential mapping grows at most linearly, is calm at the unique solution $\hat{x}$ of \eqref{eq:BK} and the collection $\{\partial f(\hat{x}), \mathcal{R}(A^{\top})\}$ is linearly regular, then there exists a constant $\theta(\hat{x})$ such that, for all $x\in \mathbb{R}^n$ and $x^* \in  \partial f(x) \cap \mathcal{R}(A^{\top})$, the following global error bound holds:
\begin{equation} \label{eq:global error bound}
    D_{f}^{x^{*}}(x, \hat{x})\leq \frac{1}{\theta(\hat{x})}\cdot \|Ax-b\|_{2}^{2}.
\end{equation}
\end{theorem}

When $f(x) = \lambda \|x\|_1 + \frac{1}{2}\|x\|^2_2$, it is obvious that $f$ is 1-strongly convex and that $\nabla f^{*}(x^{*})$ is equal to $S_{\lambda}(x^{*})$. Moreover, it is provided explicitly in \cite{schopfer2019linear} that
   \begin{equation} \label{eq:theta better form}
    \frac{1}{\theta(\hat{x})} = \frac { 1 } { \tilde { \sigma } _ { \min } ^ { 2 } ( A ) } \cdot \frac {  |\hat {  x } |_ {\min} + 2 \lambda } { | \hat { x } | _{\min} }.
\end{equation}

We are now in the position to investigate the convergence behaviour of Algorithm \ref{alg:cRABEBK}. Since the update of $x_k$ depends directly on the auxiliary sequence $\{z_k^*\}$, 
it is natural to begin by analyzing the convergence of the sequence $\{z_k^*\}$.

\begin{theorem} \label{thm:convergence RABEBK_z}
Let \( g : \mathbb{R}^m \to \mathbb{R} \) be \(\mu_g\)-strongly convex with an \(L_g\)-Lipschitz continuous gradient. If the relaxation parameter $\alpha^{(z)}$ satisfies  $$0 < \alpha^{(z)} < 2\mu_g/\beta^{\mathcal{J}}_{\max},$$  then the iterates \( \{z_k^*\} \) generated by the rRABEBK method converge in expectation to \[\hat{z} := b - \hat{y}, \qquad \text{where}\ \hat{y} \in \mathcal{R}(A),\]  which is the unique solution of the constrained optimization problem \eqref{eq:BK01}. In particular, the iterates \( \{z_k^*\} \) satisfy the following estimate:
\[
\mathbb{E} \left[ \| z^*_{k+1} - (b - \hat{y}) \|^2_2 \right] \leq \frac{2L_g}{\mu_g}\left( 1 - \frac{\alpha^{(z)}(2\mu_g -\beta_{\max}^{\mathcal{J}}\alpha^{(z)})\theta(\hat{z}) }{2\mu_g \| A \|_F^2} \right)^{k+1}D_{h}^{\tilde{z}^{*}_0}(z_0, \hat{z}).
\]
\end{theorem}
\begin{proof}
    Computing the approximate solutions $z^*_k $ amounts to applying the randomized averaging block extended Bregman-Kaczmarz method with relaxation parameter $\alpha^{(z)}$ to the constrained optimization problem \eqref{eq:BK01}. Hence, following the argument  in Section~\ref{sec:RABEBK-d}, the update scheme of \( z_k^* \) in the rRABEBK method is equivalent to
\begin{equation} \label{eq:th421}
   \begin{split}
       &\tilde{z}^*_{k+1} = \tilde{z}^*_k - \frac{\alpha^{(z)}}{\|A_{:,\mathcal{J}_{j_k}}\|^2_F} A_{:,\mathcal{J}_{j_k}}(A_{:,\mathcal{J}_{j_k}})^\top z_k, \\
       &z_{k+1} = \nabla h^*(\tilde{z}^*_{k+1}), 
   \end{split}
\end{equation}
with  initialization \( \tilde{z}^*_0 = 0 \). 

Since $g$ is $\mu_g$-strongly convex, it is known that $h(z)=g(z)-b^\top z$ is also $\mu_g$-strongly convex and $\nabla h^*$ is $1/\mu_g$-Lipschitz continuous. Recalling the definition of the Bregman distance \eqref{eq:Bregman distance},  the inequality \eqref{eq:alpha-strongly convex} implies
\begin{equation*}
\begin{split}
   D_{h}^{\tilde{z}^*_{k+1}}(z_{k+1},\hat{z})&=h^{*}(\tilde{z}^*_{k+1})- \langle \tilde{z}^*_{k+1},\hat{z} \rangle +h(\hat{z}) \\
   &\leq h^{*}(\tilde{z}^*_{k}) +\langle \nabla h^{*}(\tilde{z}^*_{k}),\tilde{z}^*_{k+1}-\tilde{z}_k^*\rangle + \frac{1}{2\mu_g}\|\tilde{z}^*_{k+1}-\tilde{z}_k^*  \|^2_2 +h(\hat{z})-\langle \tilde{z}^*_{k+1},\hat{z}\rangle,
\end{split}
\end{equation*}
where $\hat{z}$ is the unique solution of \eqref{eq:BK01}.

Using the fact $z_k = \nabla h^{*}(\tilde{z}^*_k)$, we obtain
\begin{equation*}
   \begin{split}
       D_{h}^{\tilde{z}_{k+1}^*}(z_{k+1},\hat{z}) &\leq D_{h}^{\tilde{z}_k^*}(z_k,\hat{z}) + \langle z_k, \tilde{z}_{k+1}^* - \tilde{z}_k^* \rangle + \frac{1}{2\mu_g}\|\tilde{z}_{k+1}^* - \tilde{z}_k^*\|_2^2 - \langle \tilde{z}_{k+1}^* - \tilde{z}_k^*, \hat{z} \rangle \\
       &= D_{h}^{\tilde{z}_k^*}(z_k,\hat{z}) + \langle \tilde{z}_{k+1}^* - \tilde{z}_k^*, z_k - \hat{z} \rangle + \frac{1}{2\mu_g}\|\tilde{z}_{k+1}^* - \tilde{z}_k^*\|_2^2 \\
       &= D_{h}^{\tilde{z}_k^*}(z_k,\hat{z}) - \langle \frac{\alpha^{(z)}}{\|A_{:,\mathcal{J}_{j_k}}\|_F^2} A_{:,\mathcal{J}_{j_k}} (A_{:,\mathcal{J}_{j_k}})^\top z_k, z_k - \hat{z}\rangle + \frac{1}{2\mu_g}\|\tilde{z}_{k+1}^* - \tilde{z}_k^*\|_2^2 \\
       &= D_{h}^{\tilde{z}_k^*}(z_k,\hat{z}) - \langle \frac{\alpha^{(z)}}{\|A_{:,\mathcal{J}_{j_k}}\|_F^2} (A_{:,\mathcal{J}_{j_k}})^\top z_k, (A_{:,\mathcal{J}_{j_k}})^\top z_k - (A_{:,\mathcal{J}_{j_k}})^{\top}\hat{z}\rangle + \frac{1}{2\mu_g}\|\tilde{z}_{k+1}^* - \tilde{z}_k^*\|_2^2.
   \end{split}
\end{equation*}
Recalling  that $(A_{:,\mathcal{J}_{j_k}})^{\top}\hat{z}=0$, we arrive at
\begin{align} \label{eq:conver_z_const}
D_{h}^{\tilde{z}_{k+1}^*}(z_{k+1},\hat{z}) &\leq D_{h}^{\tilde{z}_k^*}(z_k,\hat{z}) - \frac{\alpha^{(z)}}{\|A_{:,\mathcal{J}_{j_k}}\|_F^2} \|(A_{:,\mathcal{J}_{j_k}})^\top z_k\|_2^2 + \frac{1}{2\mu_g}\frac{(\alpha^{(z)})^2}{\|A_{:,\mathcal{J}_{j_k}}\|_F^4} \|A_{:,\mathcal{J}_{j_k}} (A_{:,\mathcal{J}_{j_k}})^\top z_k\|_2^2\\
   &\leq D_{h}^{\tilde{z}_k^*}(z_k,\hat{z}) -\frac{\alpha^{(z)}}{\|A_{:,\mathcal{J}_{j_k}}\|_F^2} \|(A_{:,\mathcal{J}_{j_k}})^\top z_k\|_2^2 + \frac{1}{2\mu_g}\frac{(\alpha^{(z)})^2}{\|A_{:,\mathcal{J}_{j_k}}\|_F^4}\sigma_{\max}^2(A_{:,\mathcal{J}_{j_k}} ) \|(A_{:,\mathcal{J}_{j_k}})^\top z_k\|_2^2 \notag \\ 
   &=D_{h}^{\tilde{z}_k^*}(z_k,\hat{z}) - \frac{\alpha^{(z)}(2\mu_g-\frac{\sigma_{\max}^2(A_{:,\mathcal{J}_{j_k}} ) }{\|A_{:,\mathcal{J}_{j_k}}\|_F^2}\alpha^{(z)})}{2\mu_g} \frac{\|(A_{:,\mathcal{J}_{j_k}})^\top z_k\|_2^2}{\|A_{:,\mathcal{J}_{j_k}}\|_F^2}.\label{eq:conver_z_const2}
\end{align}

It follows from the global error bound \eqref{eq:global error bound} in Theorem \ref{thm:theta} that
\begin{equation*}
   D_{h}^{\tilde{z}_k^*}(z_k,\hat{z})\leq \frac{1}{\theta(\hat{z})}\|A^\top z\|_2^2,
\end{equation*}
by which we have
\begin{equation*}
   \mathbb{E}\left[\frac{\|(A_{:,\mathcal{J}_{j_k}})^\top z_k\|_2^2}{\|A_{:,\mathcal{J}_{j_k}}\|_F^2}\right] = \sum_{j_k=1}^{t}\frac{\|A_{:,\mathcal{J}_{j_k}}\|^2_F}{\|A\|_F^2}\frac{\|(A_{:,\mathcal{J}_{j_k}})^\top z_k\|_2^2}{\|A_{:,\mathcal{J}_{j_k}}\|_F^2} = \frac{\|A^\top z\|^2_2}{\|A\|_F^2}\geq \frac{\theta(\hat{z})D_{h}^{\tilde{z}_k^*}(z_k,\hat{z})}{\|A\|^2_F}.
\end{equation*}

Recall that $$\beta^{\mathcal{J}}_{\max} = \max_{j_k \in [t]} \frac{\sigma^2_{\max} \left( A_{:,\mathcal{J}_{j_k}} \right)}{\| A_{:,\mathcal{J}_{j_k}} \|^2_F}.$$
By taking the expectation conditional on the first $k$ iterations on both sides of \eqref{eq:conver_z_const2} and invoking linearity of expectation, we obtain
\begin{equation}
\begin{split}
   \mathbb{E}\left[ D_{h}^{\tilde{z}^{*}_{k+1}}(z_{k+1}, \hat{z}) \right] &\leq \mathbb{E}\left[ D_{h}^{\tilde{z}^{*}_k}(z_k, \hat{z})  \right]- \frac{\alpha^{(z)}(2\mu_g -\beta_{\max}^{\mathcal{J}}\alpha^{(z)}) }{2\mu_g } \frac{\theta(\hat{z})}{\|A\|_F^2}\mathbb{E}\left[ D_{h}^{\tilde{z}^{*}_k}(z_k, \hat{z})  \right]\\
&= \left( 1 - \frac{\alpha^{(z)}(2\mu_g -\beta_{\max}^{\mathcal{J}}\alpha^{(z)})\theta(\hat{z}) }{2\mu_g \| A \|_F^2} \right) \mathbb{E} \left[ D_{h}^{\tilde{z}^{*}_k}(z_k, \hat{z}) \right].
\end{split}
   \end{equation}
    
Lemma \ref{lem:Euclid-Bregman}  says that the Euclidean distance provides a lower
bound for the Bregman distance. Hence,
\begin{equation}
   \mathbb{E} \left[\|z_{k+1}-\hat{z}\|_2^2\right]\leq \frac{2}{\mu_g}\left( 1 - \frac{\alpha^{(z)}(2\mu_g -\beta_{\max}^{\mathcal{J}}\alpha^{(z)})\theta(\hat{z}) }{2\mu_g \| A \|_F^2} \right)^{k+1}D_{h}^{\tilde{z}^{*}_0}(z_0, \hat{z}).
\end{equation}
Now we arrive at the conclusion
\begin{equation} \label{eq:convergence of zk_general}
   \begin{split}
        \mathbb{E} \left[\|z^*_{k+1}-(b-\hat{y})\|_2^2 \right] &=  \mathbb{E} \left[\|\nabla g(z_{k+1})-\nabla g(\hat{z})\|_2^2\right]\leq L_g  \mathbb{E} \left[\|z_{k+1}-\hat{z}\|_2^2\right] \\
        &\leq \frac{2L_g}{\mu_g}\left( 1 - \frac{\alpha^{(z)}(2\mu_g -\beta_{\max}^{\mathcal{J}}\alpha^{(z)})\theta(\hat{z}) }{2\mu_g \| A \|_F^2} \right)^{k+1}D_{h}^{\tilde{z}^{*}_0}(z_0, \hat{z}),
   \end{split}
\end{equation}
where the first inequality exploits the fact that  $g$ has a Lipschitz-continuous gradient with constant $L_g$.
\end{proof}

\begin{remark}\label{remark:zk_general}
Theorem~\ref{thm:convergence RABEBK_z} establishes that the auxiliary sequence \(\{z_k^*\}\) generated by rRABEBK converges linearly in expectation to the dual optimum \(\hat{z} = b - \hat{y}\), which is the unique solution of \eqref{eq:BK01}.  The admissible stepsize range \(0 < \alpha^{(z)} < 2\mu_g / \beta_{\max}^{\mathcal{J}}\) ensures a
sufficient decrease along each Bregman-type projection step.  
Within this range, the optimal convergence rate is attained by choosing
$$\alpha^{(z)} = 
\frac{\mu_g}{ \beta_{\max}^{\mathcal{J}}},$$ yielding
    \begin{equation*}
        \mathbb{E} \left[\|z^*_{k+1}-(b-\hat{y})\|_2^2 \right] \leq \frac{2L_g}{\mu_g}\left( 1 - \frac{\mu_g \theta(\hat{z}) }{2\beta_{\max}^{\mathcal{J}}\| A \|_F^2} \right)^{k+1}D_{h}^{\tilde{z}^{*}_0}(z_0, \hat{z}).
    \end{equation*}
Intuitively, this means that when the blocks are informative, and the stepsize is chosen appropriately, the auxiliary iterates contract toward the optimal solution at a guaranteed linear rate.
\end{remark}

Having established the convergence of the auxiliary sequence $\{z_k^*\}$, we proceed to analyze the convergence behavior of the primary iterates $\{x_k\}$ generated by Algorithm \ref{alg:cRABEBK}.
As we will see, the convergence properties of $\{z_k^*\}$ play a crucial role in controlling the primary updates and hence form the backbone of the subsequent analysis.

\begin{theorem} \label{thm:convergence RABEBK_x}
    Let \( f : \mathbb{R}^n \to \mathbb{R} \) be $\mu_f$-strongly convex, and assume that the conditions in Theorem \ref{thm:theta} hold. Then, the iterates \( \{x_k\} \) produced by Algorithm \ref{alg:cRABEBK} converge in expectation to the unique solution $\hat{x}$ of the constrained optimization problem \eqref{eq:BK02}. More precisely,
for any $\varepsilon > 0 $ and  $0 < \alpha^{(x)} \leq \mu_f/\beta^{\mathcal{I}}_{\max}
$,  there exist constants \( c_1,c_2\) such that
\begin{equation}\label{eq:thm_x_convergence_general}
\begin{split}
   \mathbb{E} \left[D_{f}^{x_{k+1}^{*}}(x_{k+1},\hat{x})\right] \leq &\left(1-c_1\cdot\gamma(\hat{x})\right)\mathbb{E}\left[D_{f}^{x_{k}^{*}}(x_{k},\hat{x})\right] \\&+c_2\cdot \frac{2L_g}{\mu_g}\left( 1 - \frac{\alpha^{(z)}(2\mu_g -\beta_{\max}^{\mathcal{J}}\alpha^{(z)})\theta(\hat{z}) }{2\mu_g \| A \|_F^2} \right)^{k+1}D_{h}^{\tilde{z}^{*}_0}(z_0, \hat{z}),
\end{split}
\end{equation}
where $$c_1=\frac{\alpha^{(x)}(2\mu_f-\beta^{\mathcal{I}}_{\max}\alpha^{(x)})-\frac{1}{\varepsilon}\alpha^{(x)}(\mu_f-\beta^{\mathcal{I}}_{\max}\alpha^{(x)})}{2\mu_f\|A\|^2_F},$$
and
$$c_2 = \frac{(\alpha^{(x)})^2\beta^{\mathcal{I}}_{\max}+\varepsilon\alpha^{(x)}(\mu_f-\beta^{\mathcal{I}}_{\max}\alpha^{(x)})}{2\mu_f\|A\|^2_F}.$$ 
\end{theorem}
\begin{proof}
Let $v_k := A_{\mathcal{I}_{i_k},:} x_k - b_{\mathcal{I}_{i_k}} + z^{*}_{k+1,\mathcal{I}_{i_k}}
$. Since $\nabla f^*$ is $1/\mu_f$-Lipschitz continuous,  from the definition of the Bregman distance \eqref{eq:Bregman distance} and the smoothness bound \eqref{eq:alpha-strongly convex} of $f^*$, we obtain
\[
f^*(x_k^* + s) \le f^*(x_k^*) + \langle \nabla f^*(x_k^*), s \rangle + \frac{1}{2\mu_f}\|s\|_2^2
\quad\text{for all } s\in\mathbb{R}^n.
\]
Taking \(s=-\dfrac{\alpha^{(x)}}{\|A_{\mathcal{I}_{i_k},:}\|_F^2}(A_{\mathcal{I}_{i_k},:})^\top v_k\) yields the following inequality 
\begin{align}
   D_{f}^{x_{k+1}^{*}}(x_{k+1},\hat{x})&=f^{*}(x_{k+1}^{*})- \langle x_{k+1}^{*},\hat{x} \rangle +f(\hat{x}) \notag \\
   &=f^*\left(x_k^*-\frac{\alpha^{(x)}}{\|A_{\mathcal{I}_{i_k},:}\|^2_F}(A_{\mathcal{I}_{i_k},:})^\top v_k\right)-\langle x_k^*,\hat{x}\rangle+\left\langle \frac{\alpha^{(x)}}{\|A_{\mathcal{I}_{i_k},:}\|^2_F}(A_{\mathcal{I}_{i_k},:})^\top v_k, \hat{x}\right\rangle +f(\hat{x})\notag\\
   &\leq f^{*}(x_{k}^{*}) +\langle \nabla f^{*}(x^*_k),-\frac{\alpha^{(x)}}{\|A_{\mathcal{I}_{i_k},:}\|^2_F}(A_{\mathcal{I}_{i_k},:})^\top v_k\rangle + \frac{1}{2\mu_f}\left\|-\frac{\alpha^{(x)}}{\|A_{\mathcal{I}_{i_k},:}\|^2_F}(A_{\mathcal{I}_{i_k},:})^\top v_k  \right\|^2_2 \notag\\ 
   &\quad-\langle x^*_{k},\hat{x}\rangle+f(\hat{x})+\left\langle \frac{\alpha^{(x)}}{\|A_{\mathcal{I}_{i_k},:}\|^2_F} v_k, A_{\mathcal{I}_{i_k},:}\hat{x}\right\rangle.\label{eq:converge_x_const2}
\end{align}
Using the identities $x_k = \nabla f^{*}(x^*_k)$ and $A_{\mathcal{I}_{i_k}} \hat{x} = \hat{y}_{\mathcal{I}_{i_k}}$, and noting that  $\beta^{\mathcal{I}}_{\max} = \max_{i_k \in [s]} \frac{\sigma^2_{\max} \left( A_{\mathcal{I}_{i_k},:} \right)}{\| A_{\mathcal{I}_{i_k},:} \|^2_F}$, the inequality \eqref{eq:converge_x_const2} reaches
\begin{align}\label{eq:converge_x_const}
    D_{f}^{x_{k+1}^{*}}(x_{k+1},\hat{x})& \leq  D_{f}^{x_{k}^{*}}(x_{k},\hat{x}) - \frac{\alpha^{(x)}}{\|A_{\mathcal{I}_{i_k},:}\|^2_F}\left\langle  v_k, A_{\mathcal{I}_{i_k},:}x_k - \hat{y}_{\mathcal{I}_{i_k}}\right\rangle + \frac{1}{2\mu_f}\frac{(\alpha^{(x)})^2}{\|A_{\mathcal{I}_{i_k},:}\|^4_F}\|(A_{\mathcal{I}_{i_k},:})^\top v_k\|^2_2   \\
   &  \leq  D_{f}^{x_{k}^{*}}(x_{k},\hat{x}) - \frac{\alpha^{(x)}}{\|A_{\mathcal{I}_{i_k},:}\|^2_F}\left\langle  v_k, A_{\mathcal{I}_{i_k},:}x_k - \hat{y}_{\mathcal{I}_{i_k}}\right\rangle + \frac{1}{2\mu_f}\frac{(\alpha^{(x)})^2}{\|A_{\mathcal{I}_{i_k},:}\|^4_F}\sigma_{\max}^2(A_{\mathcal{I}_{i_k},:})\| v_k\|^2_2 \notag\\
   & \leq D_{f}^{x_{k}^{*}}(x_{k},\hat{x}) - \frac{\alpha^{(x)}}{\|A_{\mathcal{I}_{i_k},:}\|^2_F}\left\langle  v_k, A_{\mathcal{I}_{i_k},:}x_k - \hat{y}_{\mathcal{I}_{i_k}}\right\rangle + \frac{1}{2\mu_f}\frac{(\alpha^{(x)})^2}{\|A_{\mathcal{I}_{i_k},:}\|^2_F}\beta^{\mathcal{I}}_{\max} \| v_k\|^2_2.\label{eq:converge_x_const4}
\end{align}
Using
\begin{equation*}
   \begin{split}
       \left\langle  v_k, A_{\mathcal{I}_{i_k},:}x_k - \hat{y}_{\mathcal{I}_{i_k}}\right\rangle         &=\left\langle A_{\mathcal{I}_{i_k},:} x_k - b_{\mathcal{I}_{i_k}} + z^{*}_{k+1,\mathcal{I}_{i_k}} -\hat{y}_{\mathcal{I}_{i_k}}+\hat{y}_{\mathcal{I}_{i_k}}, A_{\mathcal{I}_{i_k},:}x_k - \hat{y}_{\mathcal{I}_{i_k}}\right\rangle \\
       &=\|A_{\mathcal{I}_{i_k},:}x_k - \hat{y}_{\mathcal{I}_{i_k}}\|^2_2+\left\langle \hat{y}_{\mathcal{I}_{i_k}} - b_{\mathcal{I}_{i_k}} + z^{*}_{k+1,\mathcal{I}_{i_k}}, A_{\mathcal{I}_{i_k},:}x_k - \hat{y}_{\mathcal{I}_{i_k}}\right\rangle,
   \end{split}
\end{equation*}
and 
\begin{equation} \label{eq:convergence_vk}
   \|v_k\|^2_2 = \|A_{\mathcal{I}_{i_k},:}x_k - \hat{y}_{\mathcal{I}_{i_k}}\|^2_2 + 2 \left\langle \hat{y}_{\mathcal{I}_{i_k}} - b_{\mathcal{I}_{i_k}} + z^{*}_{k+1,\mathcal{I}_{i_k}}, A_{\mathcal{I}_{i_k},:}x_k - \hat{y}_{\mathcal{I}_{i_k}}\right\rangle + \|z^*_{k+1,\mathcal{I}_{{i_k}}}-(b_{\mathcal{I}_{i_k}}-\hat{y}_{\mathcal{I}_{i_k}})\|^2_2,
\end{equation}
the inequality \eqref{eq:converge_x_const4} becomes
\begin{equation}\label{eq:preconvergence of x}
\begin{split}
   D_{f}^{x_{k+1}^{*}}(x_{k+1},\hat{x}) \leq D_{f}^{x_{k}^{*}}(x_{k},\hat{x}) &- \frac{\alpha^{(x)}(2\mu_f-\beta^{\mathcal{I}}_{\max}\alpha^{(x)})}{2\mu_f\|A_{\mathcal{I}_{i_k},:}\|^2_F}\left\|A_{\mathcal{I}_{i_k},:}x_k - \hat{y}_{\mathcal{I}_{i_k}}\right\|^2_2 \\
   &-\frac{\alpha^{(x)}(\mu_f-\beta^{\mathcal{I}}_{\max}\alpha^{(x)})}{\mu_f\|A_{\mathcal{I}_{i_k},:}\|^2_F}\left\langle \hat{y}_{\mathcal{I}_{i_k}} - b_{\mathcal{I}_{i_k}} + z^*_{k+1,\mathcal{I}_{{i_k}}},A_{\mathcal{I}_{i_k},:}x_k - \hat{y}_{\mathcal{I}_{i_k}}\right\rangle \\
&+\frac{(\alpha^{(x)})^2\beta^{\mathcal{I}}_{\max}}{2\mu_f\|A_{\mathcal{I}_{i_k},:}\|_F^2}\|z^*_{k+1,\mathcal{I}_{{i_k}}}-(b_{\mathcal{I}_{i_k}}-\hat{y}_{\mathcal{I}_{i_k}})\|^2_2.
\end{split} 
\end{equation}
Note that for any $\varepsilon >0$, we have
\begin{equation} \label{eq:Du2020_2.5}
\begin{split}
   2\left\langle \hat{y}_{\mathcal{I}_{i_k}} - b_{\mathcal{I}_{i_k}} + z^{*}_{k+1,\mathcal{I}_{i_k}}, A_{\mathcal{I}_{i_k},:}x_k - \hat{y}_{\mathcal{I}_{i_k}}\right\rangle \geq &-\varepsilon \|\hat{y}_{\mathcal{I}_{i_k}} - b_{\mathcal{I}_{i_k}} + z^*_{k+1,\mathcal{I}_{{i_k}}}\|^2_2 \\
   &-\frac{1}{\varepsilon}\|A_{\mathcal{I}_{i_k},:}x_k - \hat{y}_{\mathcal{I}_{i_k}}\|_2^2.
\end{split}
\end{equation}
Combining inequality \eqref{eq:preconvergence of x} and \eqref{eq:Du2020_2.5} yields
\begin{equation}\label{eq:conver_z_const5} 
   \begin{split}
       D_{f}^{x_{k+1}^{*}}(x_{k+1},\hat{x}) \leq D_{f}^{x_{k}^{*}}(x_{k},\hat{x}) &- \left( \frac{\alpha^{(x)}(2\mu_f-\beta^{\mathcal{I}}_{\max}\alpha^{(x)})}{2\mu_f\|A_{\mathcal{I}_{i_k},:}\|^2_F}-\frac{\frac{1}{\varepsilon}\alpha^{(x)}(\mu_f-\beta^{\mathcal{I}}_{\max}\alpha^{(x)})}{2\mu_f\|A_{\mathcal{I}_{i_k},:}\|^2_F}\right)\left\|A_{\mathcal{I}_{i_k},:}x_k - \hat{y}_{\mathcal{I}_{i_k}}\right\|^2_2\\
       &+\left( \frac{(\alpha^{(x)})^2\beta^{\mathcal{I}}_{\max}}{2\mu_f\|A_{\mathcal{I}_{i_k},:}\|^2_F}+\frac{\varepsilon\alpha^{(x)}(\mu_f-\beta^{\mathcal{I}}_{\max}\alpha^{(x)})}{2\mu_f\|A_{\mathcal{I}_{i_k},:}\|^2_F}\right)\|z^*_{k+1,\mathcal{I}_{{i_k}}}-(b_{\mathcal{I}_{i_k}}-\hat{y}_{\mathcal{I}_{i_k}})\|^2_2,
   \end{split}
\end{equation}
whenever $\mu_f \geq \alpha^{(x)}\beta_{\max}^{\mathcal{I}}$. 
Taking the conditional expectation over the first $k$ iterations on both sides of \eqref{eq:conver_z_const5} and using linearity of expectation, we have
\begin{equation} \label{eq:convergence_x_general_last}
   \mathbb{E}\left[D_{f}^{x_{k+1}^{*}}(x_{k+1},\hat{x})\right]\leq \mathbb{E}\left[D_{f}^{x_{k}^{*}}(x_{k},\hat{x})\right] - c_1 \cdot \mathbb{E}\left[\|Ax_k - \hat{y}\|_2^2\right]+c_2\cdot \mathbb{E}\left[\|z_{k+1}^* - (b-\hat{y})\|_2^2\right].
\end{equation}

It follows from the inequality \eqref{eq:global error bound} in Theorem \ref{thm:theta} that $D_{f}^{x_{k}^{*}}(x_{k},\hat{x})\leq \frac{1}{\gamma(\hat{x})}\|Ax_k-\hat{y}\|_2^2$, therefore
\begin{equation}
\begin{split}
   \mathbb{E}\left[D_{f}^{x_{k+1}^{*}}(x_{k+1},\hat{x})\right]&\leq \left(1-c_1\cdot\gamma(\hat{x})\right)\mathbb{E}\left[D_{f}^{x_{k}^{*}}(x_{k},\hat{x})\right] +c_2\cdot \mathbb{E}\left[\|z_{k+1}^* - (b-\hat{y})\|_2^2\right] \\
   &\leq \left(1-c_1\cdot\gamma(\hat{x})\right)\mathbb{E}\left[D_{f}^{x_{k}^{*}}(x_{k},\hat{x})\right] \\ &\quad+c_2\cdot \frac{2L_g}{\mu_g}\left( 1 - \frac{\alpha_k^{(z)}(2\mu_g -\beta_{\max}^{\mathcal{J}}\alpha_k^{(z)})\theta(\hat{z}) }{2\mu_g \| A \|_F^2} \right)^{k+1}D_{h}^{\tilde{z}^{*}_0}(z_0, \hat{z}),
\end{split}
\end{equation}
where the second inequality follows from \eqref{eq:convergence of zk_general}. 

Since $c_2$ is a constant and the second term tends to $0$ as $k \to \infty$, the iterates $\{x_k\}$ converge in expectation to the unique solution $\hat{x}$. Hence, the proof is complete.
\end{proof}

\begin{remark} \label{remark:xk_general}
 The relaxation parameter \(\alpha^{(x)}\) governs the step size taken
along the averaged Bregman-Kaczmarz update direction for the primary variable. From \eqref{eq:thm_x_convergence_general}, the linear convergence rate of
the primary sequence \(\{x_k\}\) with relaxation parameter \(\alpha^{(x)}\) is
\begin{equation} \label{eq:11-8-1}
1-\frac{\alpha^{(x)}(2\mu_f-\beta^{\mathcal{I}}_{\max}\alpha^{(x)})-\frac{1}{\varepsilon}\alpha^{(x)}(\mu_f-\beta^{\mathcal{I}}_{\max}\alpha^{(x)})}{2\mu_f\|A\|^2_F}\cdot \gamma(\hat{x}),
\end{equation}
for any $\varepsilon>0$.  When \(\varepsilon\) is taken sufficiently large, the dominant term controlling the rate becomes
\(\alpha^{(x)}(2\mu_f - \beta^{\mathcal{I}}_{\max}\alpha^{(x)})\),
which attains its maximum at
 \begin{equation} \label{eq:11-23-1}
    \alpha^{(x)} = \frac{\mu_f}{\beta^{\mathcal{I}}_{\max}}.
\end{equation}
At this choice, the additional term
\(\frac{1}{\varepsilon}\alpha^{(x)}(\mu_f - \beta^{\mathcal{I}}_{\max}\alpha^{(x)})\)
vanishes, and the contraction factor \eqref{eq:11-8-1} simplifies to
\begin{equation} \label{eq:11-6-1}
    1 - \frac{\mu_f}{2\,\beta^{\mathcal{I}}_{\max}\,\|A\|_F^2}\cdot \gamma(\hat{x}).
\end{equation}
Hence, the choice \eqref{eq:11-23-1}  yields the optimal contraction factor for the primary sequence in expectation.
\end{remark}

 In particular, when considering problem \eqref{eq:main_partical} with the objective function
\[
f(x) = \lambda \|x\|_1 + \frac{1}{2} \|x\|_2^2,
\]
the quantity $\gamma(\hat{x})$ admits an explicit expression, as shown in \eqref{eq:theta better form}. Therefore, for the sparse least‐squares problem, the proposed method yields the following explicit estimate of the convergence factor:
\[
1 - \frac{\mu_f}{2 \beta^{\mathcal{I}}_{\max} \|A\|_F^2} \cdot \frac{1}{\tilde{\sigma}_{\min}^2(A)} \cdot \frac{|\hat{x}|_{\min} + 2\lambda}{|\hat{x}|_{\min}}.
\]

It is worth noting that the dominant convergence rate of the proposed RABEBK method is given by \eqref{eq:11-6-1}, whereas that of the REBK method is
    $$1-\frac{\mu_f}{2\|A\|^2_F }\cdot \gamma(\hat{x}).$$ 
By \eqref{eq:25-23-2}, we have $\beta^{\mathcal{I}}_{\max} < 1$. Consequently, the dominant contraction factor of the proposed rRABEBK method is strictly smaller than that of the REBK method \cite{Schpfer2022ExtendedRK}. This establishes that our rRABEBK method achieves a provably faster theoretical convergence rate than REBK.


\section{Numerical experiments} \label{sec:numerical experiments}
This section presents a set of numerical experiments to assess the efficiency and robustness of our rRABEBK method. We focus on sparse least-squares problems, a representative setting that highlights the ability of methods to handle both data inconsistency and sparsity-promoting regularization. 

In the considered experiments, the objective function is given by
\[
f(x) = \lambda \|x\|_{1} + \frac{1}{2}\|x\|_{2}^{2},
\]
and the data fidelity term is modeled as
\[
g^{*}(b - y) = \frac{1}{2}\|b - y\|_{2}^{2}.
\]
Under this setting, the gradient mappings in Algorithm~\ref{alg:cRABEBK} admit simple closed-form expressions. In Step~5, since 
\( \nabla g^{*}(z_k^{*}) = z_k^{*} \), we directly obtain \( z_{k+1} = z_{k+1}^{*} \); in Step~8, the primal  update reduces to the proximal operation
\[
x_{k+1} = S_{\lambda}(x_{k+1}^{*}),
\]
where \( S_{\lambda}(\cdot) \) denotes the soft-thresholding operator.

The sparse solution $\hat{x} \in \mathbb{R}^n$ is generated as a random vector with $s$ nonzero entries drawn from the standard normal distribution. We set the sparsity parameter to $\lambda = 5$, and choose the number of nonzero entries of exact solution as  $ s =\lceil 0.01\,n\rceil$. The corresponding right-hand side is given by
$$
b = A\hat{x} + e,
$$
where $e \in \mathcal{N}(A^\top)$. 
Such a vector $e$ can be obtained via the \texttt{MATLAB} function \verb|null|. 
Specifically, we construct the noise as $e = H v$, where the columns of $H$ form an orthonormal basis of $\mathcal{N}(A^\top)$, and $v$ is a random vector sampled uniformly from the sphere $\partial B_\rho(0)$ of radius 
$$
\rho = q \| \hat{b} \|_2,
\qquad \hat{b} = A\hat{x},
$$
with $q$ denoting the relative noise level. 
In our experiments, we set $q = 5$.

For all algorithms, we use the initialization $z_0=b$ and $x_0^*=0$. Each algorithm is terminated when either the iteration count exceeds $5\times 10^6$ or the relative error (ERR), which is defined as 
$$ \mathrm{ERR}=\frac{\left\|x_{k}-\hat{x}\right\|_2}{\|\hat{x}\|_2},$$ falls below $10^{-5}$.

Our algorithm relies on a block structure, let $\tau$ denote the block size. We assume that the subsets $\{\mathcal{I}_i\}_{i=1}^{s-1}$ and $\{\mathcal{J}_j\}_{j=1}^{t-1}$ have the same size $\tau$, i.e., $|\mathcal{I}_i| = |\mathcal{J}_j| = \tau$. The row index set $\{1,\ldots,m\}$ is partitioned into $\{\mathcal{I}_i\}_{i=1}^s$:
\begin{equation*}
   \begin{split}
       \mathcal{I}_i &= \{(i-1)\tau + 1,\, (i-1)\tau + 2,\, \ldots,\, i\tau\}, \quad i = 1,2,\ldots,s-1, \\
       \mathcal{I}_s &= \{(s-1)\tau + 1,\, (s-1)\tau + 2,\, \ldots,\, m\}, \qquad |\mathcal{I}_s| \leq \tau.
   \end{split}
\end{equation*}
Similarly, the column index set $\{1,\ldots,n\}$ is partitioned into $\{\mathcal{J}_j\}_{j=1}^t$:
\begin{equation*}
   \begin{split}
       \mathcal{J}_j &= \{(j-1)\tau + 1,\, (j-1)\tau + 2,\, \ldots,\, j\tau\}, \quad j = 1,2,\ldots,t-1,\\
       \mathcal{J}_t &= \{(t-1)\tau + 1,\, (t-1)\tau + 2,\, \ldots,\, n\}, \qquad |\mathcal{J}_t| \leq \tau.
   \end{split}
\end{equation*}
In all experiments, the block size of the proposed rRABEBK method is fixed at 
$\tau = 20$, which provides a reasonable compromise between computational cost 
and convergence efficiency. 

Unless otherwise stated, we set identical relaxation 
parameters for the auxiliary and primary updates, namely,
\[
\alpha^{(z)} = \alpha^{(x)} = \alpha. 
\]

We compare the proposed rRABEBK method with several state-of-the-art Kaczmarz-type algorithms. 
The baseline includes the randomized extended Bregman-Kaczmarz (REBK) method and one of its block extensions. In addition, a randomized averaging block extended Kaczmarz method without the Bregman update step is also included for comparison. The algorithms under consideration are summarized as follows:
\begin{enumerate}
    \item \textbf{REBK} \cite{Schpfer2022ExtendedRK}: the randomized extended Bregman-Kaczmarz method.
    \item \textbf{RABEBK}: the randomized averaging block extended Bregman-Kaczmarz method, corresponding to the proposed rRABEBK scheme with relaxation parameter $\alpha = 1$.
  \item \textbf{RABEK} \cite{du2020randomized}: the randomized averaging block extended Kaczmarz method.
\end{enumerate}
To further investigate the influence of the relaxation parameter,  we consider three variants of the proposed rRABEBK method, which differ only in the choice of $\alpha$. The parameter is scaled relative to the theoretical baseline $1/\beta_{\max}$, where $ \beta_{\max} := \max\bigl(\beta^{\mathcal{I}}_{\max},\,  \beta^{\mathcal{J}}_{\max}\bigr)$: 

\begin{enumerate}
\setcounter{enumi}{3}
    \item \textbf{rRABEBK-1}:  rRABEBK with relaxation parameter $\alpha = 1/\beta_{\max}$.
    \item \textbf{rRABEBK-1.75}:  rRABEBK with relaxation parameter $\alpha = 1.75/\beta_{\max}$.
    \item \textbf{rRABEBK-2.25}:  rRABEBK with relaxation parameter $\alpha = 2.25/\beta_{\max}$.
\end{enumerate}

These comparisons allow us to assess both the baseline performance and the sensitivity of rRABEBK with respect to the relaxation parameter.

All experiments are carried out in \texttt{MATLAB} on a computer equipped with an Intel Core i7 processor and 16\,GB of RAM.

\subsection{Experiments on Bregman Updates and the Necessity of Relaxation}
The goal of this work is to develop an efficient algorithm capable of addressing inconsistent linear systems while simultaneously incorporating prior structural information, such as sparsity in the solution. To this end, we propose the relaxed randomized averaging block extended Bregman-Kaczmarz (rRABEBK) method. The purpose of this experiment is to clarify the respective roles of the
Bregman update and the relaxation parameter in the proposed framework. By comparing the six algorithmic variants introduced at the beginning of this section, we are able to quantify the effect of (i) incorporating Bregman corrections and (ii) introducing a tunable relaxation parameter on convergence speed and reconstruction accuracy in sparse recovery problems.

We consider sparse recovery tasks in which the coefficient matrix is a Bernoulli matrix of size $500 \times 1000$, whose entries are independently sampled from the symmetric Bernoulli distribution taking values in $\{+1,-1\}$ with equal probability. Such matrices are widely used in compressive sensing due to their binary structure and favourable incoherence properties with respect to canonical sparsifying transforms such as the Fourier or wavelet basis. These matrices satisfy the Restricted Isometry Property (RIP) with high probability \cite{candes2006near}, thus providing a reliable testbed for evaluating sparse reconstruction algorithms.

\begin{figure}[!htbp]
    \centering
    \includegraphics[width=0.8\linewidth]{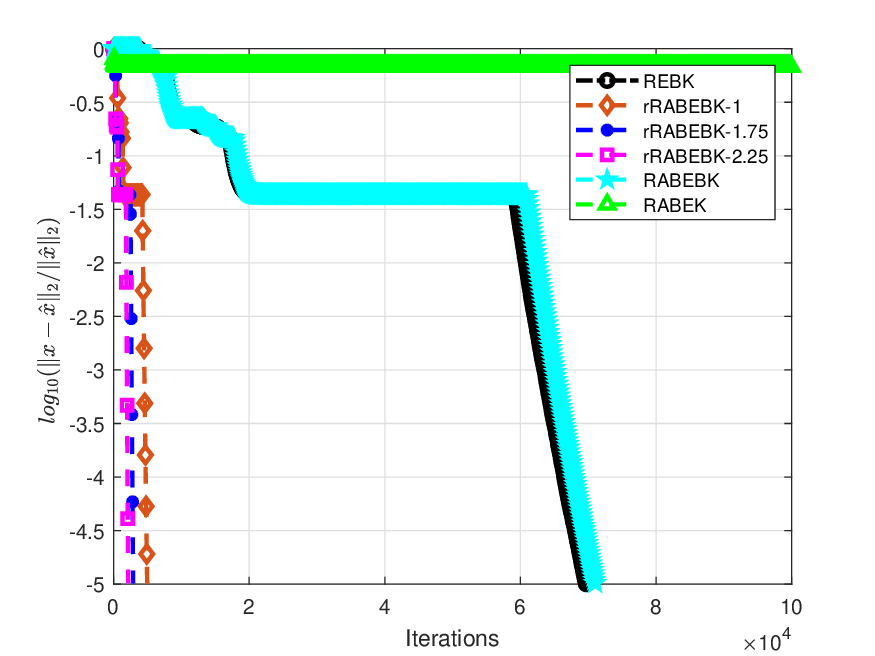}
    \caption{Convergence histories of REBK, RABEK, RABEBK, and the relaxed rRABEBK variants 
for a Bernoulli coefficient matrix of size $500\times 1000$.}
    \label{fig:ber}
\end{figure}

Figure \ref{fig:ber} reports the convergence curves of all six algorithms under
comparison. Several key observations can be made.

\emph{(i) Necessity of the Bregman update.}
The randomized averaging block extended Kaczmarz (RABEK) method \cite{du2020randomized}, which does not incorporate the Bregman update, fails to converge to the true sparse signal. As seen in Figure \ref{fig:ber}, the method stagnates at a relatively large error level. This behaviour is expected: without the soft-thresholding operator $S_\lambda = \nabla f^{*}$ induced by the Bregman update, the iterates naturally approach the minimum-norm solution of the underdetermined linear system rather than a sparse minimizer. This clearly demonstrates that the Bregman correction is indispensable for promoting sparsity in the reconstruction process.

\emph{(ii) Necessity of relaxation.} Although the unrelaxed block Bregman method RABEBK (corresponding to rRABEBK with $\alpha = 1$) converges to the correct sparse solution, its convergence speed is noticeably slower than that of the relaxed variants. This indicates that appropriate relaxation is crucial to  exploit the benefits of block updates and to accelerate convergence.

\emph{(iii) Effectiveness of the proposed relaxed variants.}
The relaxed variants rRABEBK-1, rRABEBK-1.75, and rRABEBK-2.25 deliver markedly faster convergence than both REBK and RABEBK. In particular, a moderately aggressive relaxation (e.g., $\alpha = 1.75/\beta_{\max}$) provides the most consistent acceleration, while a more aggressive choice (e.g., $\alpha = 2.25/\beta_{\max}$) can further speed up convergence but may introduce mild oscillations in the early iterations. These results confirm that the relaxation parameter provides an effective mechanism for balancing convergence speed and numerical stability.

In summary, Figure \ref{fig:ber} demonstrates that both components, the Bregman update and an appropriately chosen relaxation parameter, are essential for achieving fast and robust sparse recovery in inconsistent systems. Consequently, in the subsequent experiments we focus on the relaxed variants of RABEBK and no longer include RABEK and the unrelaxed RABEBK in the comparisons, as they are demonstrably less effective in this setting.

\subsection{Numerical Experiments with Gaussian Random Matrices}
In this experiment, we evaluate the performance of the proposed rRABEBK method on sparse least-squares problems with Gaussian coefficient matrices, whose entries are independently sampled from the standard normal distribution using the \texttt{MATLAB} function \texttt{randn}. Compared with Bernoulli matrices, Gaussian matrices exhibit better incoherence and more favourable spectral concentration, making them a canonical benchmark for compressed sensing and randomized numerical linear algebra. Our goal is to investigate the convergence behaviour and robustness of the rRABEBK framework under these well-conditioned settings.

For the six problem sizes under consideration, Table~\ref{table:Guassian} reports the number of iterations (IT) and the CPU time in seconds (CPU) required by REBK as well as by the proposed rRABEBK variants equipped with different relaxation parameters.

\begin{table}[!htbp] 
\scriptsize 
\centering 
\caption{Numerical performance of REBK and rRABEBK variants on sparse least-squares problems with Gaussian coefficient matrices.}\label{table:Guassian}  
\begin{tabular}{|c|c|c|c|c|c|c|c|c|c|}
\hline
\multirow{2}{*} {$m\times n$}
&\multicolumn{2}{c}{REBK} &\multicolumn{2}{|c}{rRABEBK-1} &\multicolumn{2}{|c}{rRABEBK-1.75} &\multicolumn{2}{|c|}{rRABEBK-2.25} \\ 
\cline{2-9} 
& IT &CPU  &IT &CPU &IT &CPU &IT &CPU \\ 
\hline  
$1000\times 500$   &	 90624&	 4.525&	  6795&	1.610&	 3948&	1.123  &	 3203&	0.815\\  
\hline  
$500\times 1000$   &	 55189&	 2.378&	 	 4402&	 0.894&	 2335&	0.446 &	 1616&	0.330  \\ 
\hline  
$2000\times 1000$  &	 331375&	154.141 &	 	 22066&	16.987&	 11689&	8.956 &	 9937&	7.616  \\  
\hline  
$1000\times 2000$   &	 698962&	 571.494&	 	 46252&	81.566&	 25219&	46.033 &	 17920&	28.517  \\  
\hline  
$4000\times 2000$   &	 195094&	 977.846&	 	 11870&	64.467&	 6175&	32.204 &	 7944&	41.848  \\   
\hline  
$2000\times 4000$   &	 906598&	 3744.716&	 	 55125&	267.296&	 36599&	 188.740 & 26475&	 146.569 \\   
\hline  
\end{tabular}  
\end{table} 

Table \ref{table:Guassian} demonstrates the clear advantages of the rRABEBK methods over the REBK algorithm across all dimensions considered. In every test case, all rRABEBK variants reduce both the iteration count and CPU time by substantial margins, often by an order of magnitude. This acceleration can be attributed to two key components:  (i) the averaging block update, which incorporates information from multiple hyperplanes at each iteration; and  (ii) the relaxation strategy, which effectively amplifies the descent direction and enhances the overall contraction behaviour.

A consistent trend observed in the results is that larger relaxation parameters (e.g., rRABEBK-1.75 and rRABEBK-2.25) generally lead to faster convergence than rRABEBK-1. For example, in the $1000 \times 2000$ case, the CPU time drops from $81.566$ seconds (rRABEBK-1) to $46.033$ seconds (rRABEBK-1.75), and further to $28.517$ seconds (rRABEBK-2.25). This confirms the benefit of tuning the relaxation factor to strengthen the projection step and accelerate error reduction.

However, the performance does not always improve monotonically with increasing relaxation. A representative instance appears in the $4000 \times 2000$ test: rRABEBK-1.75 achieves the best performance with a CPU time of $32.204$ seconds, outperforming rRABEBK-2.25, which takes $41.848$ seconds. This behaviour is consistent with the theoretical observation that overly aggressive relaxation may lead to overshooting, slightly reducing efficiency in certain problem settings.

\begin{figure}[!htbp] 
\centering
	\subfigure[$m=1000,n=500$]   
	{
		\begin{minipage}[t]{0.31\linewidth}
			\centering
			\includegraphics[width=1\textwidth]{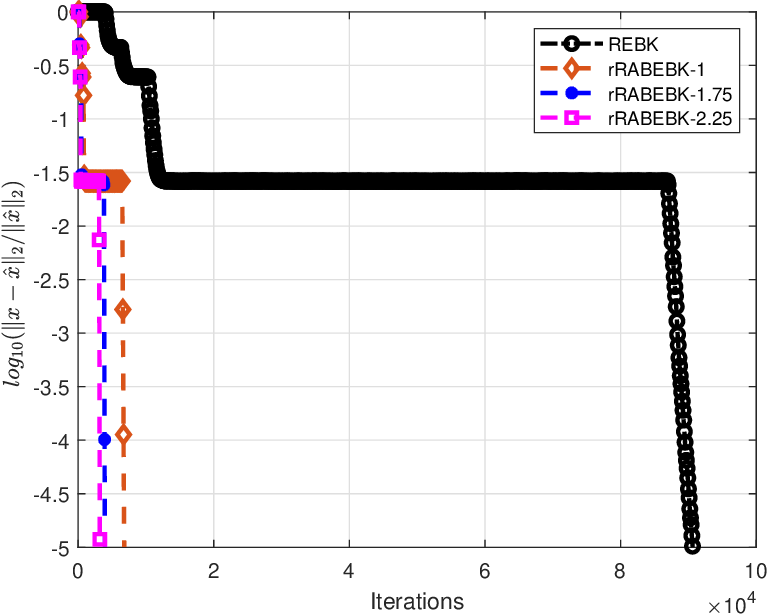}
		\end{minipage}
	}
	\subfigure[$m=2000,n=1000$]  
	{
		\begin{minipage}[t]{0.31\linewidth}
			\centering
			\includegraphics[width=1\textwidth]{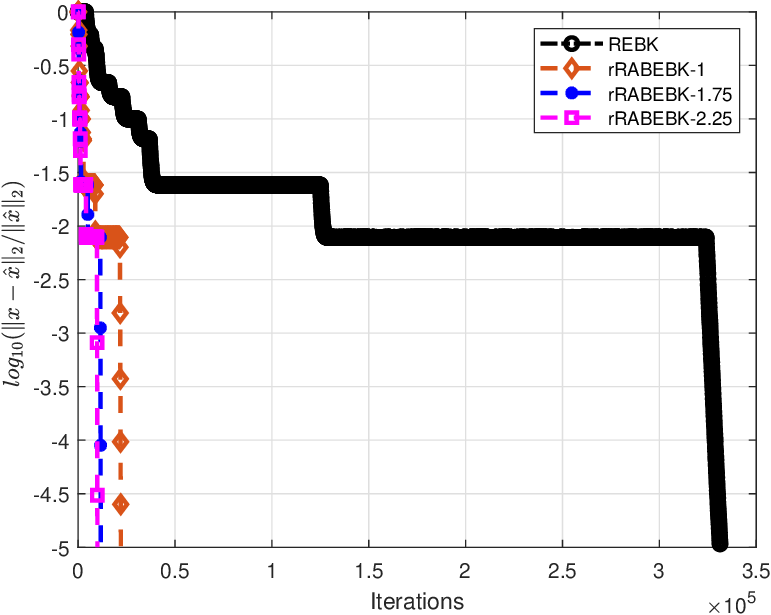}
		\end{minipage}
    }
    \subfigure[$m=4000,n=2000$]   
	{
		\begin{minipage}[t]{0.31\linewidth}
			\centering
			\includegraphics[width=1\textwidth]{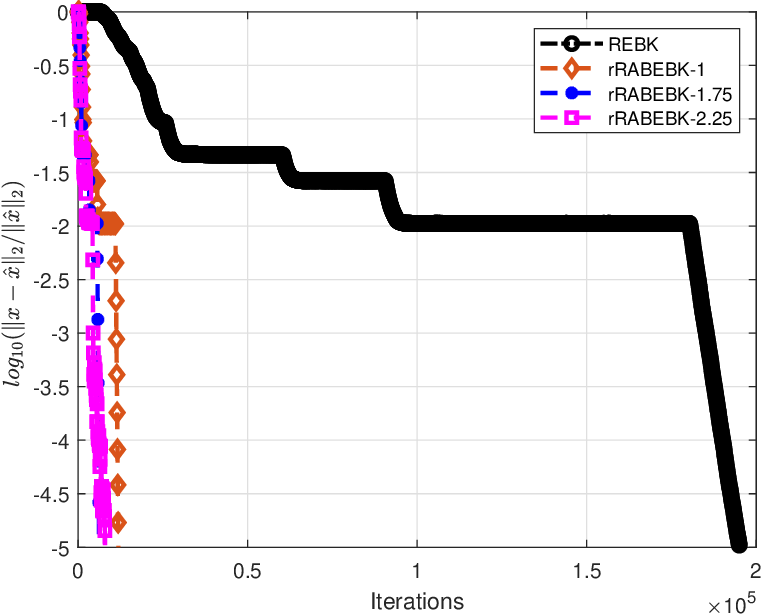}
		\end{minipage}
	}

	\subfigure[$m=500,n=1000$]   
	{
		\begin{minipage}[t]{0.31\linewidth}
			\centering
			\includegraphics[width=1\textwidth]{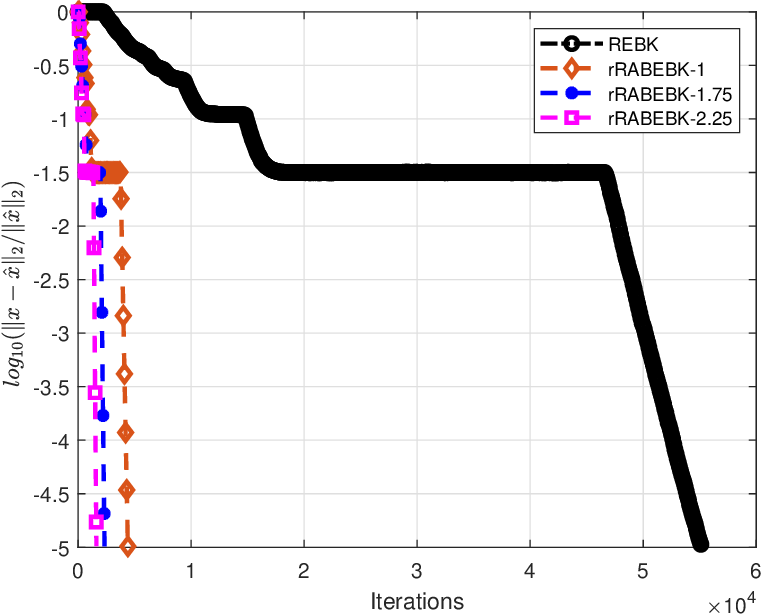}
		\end{minipage}
	}
	\subfigure[$m=1000,n=2000$]  
	{
		\begin{minipage}[t]{0.31\linewidth}
			\centering
			\includegraphics[width=1\textwidth]{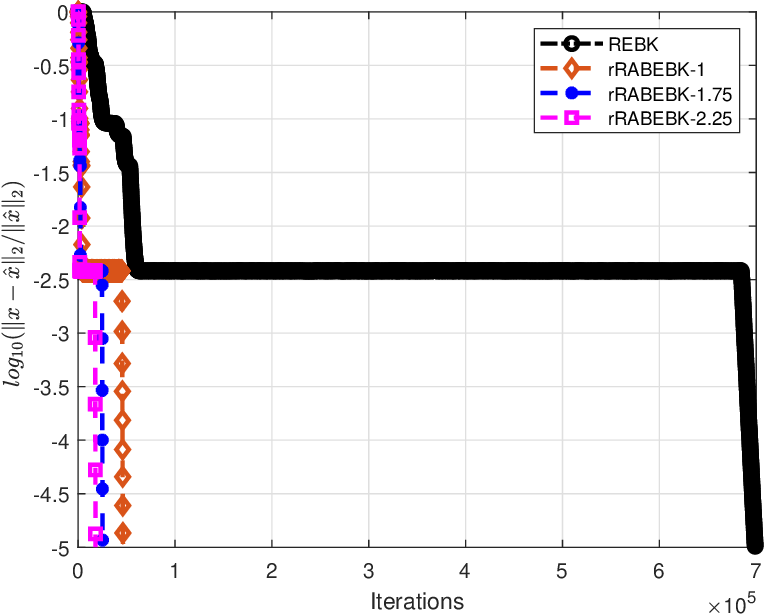}
		\end{minipage}
    }
    \subfigure[$m=2000,n=4000$]   
	{
		\begin{minipage}[t]{0.31\linewidth}
			\centering
			\includegraphics[width=1\textwidth]{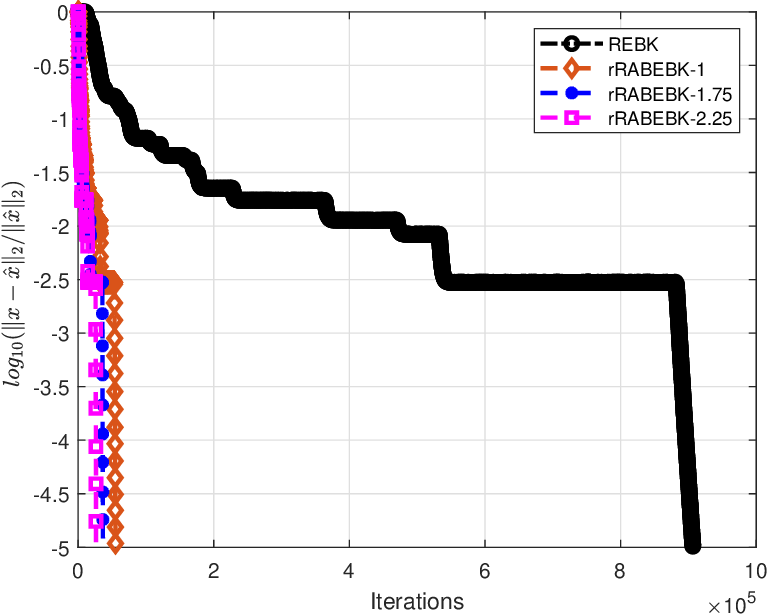}
		\end{minipage}
	}

\caption{Relative error as a function of the iteration count for Gaussian test matrices. The top panel reports results for overdetermined systems, whereas the bottom panel presents those for underdetermined systems.} 	
		\label{fig:Guassian}
\end{figure}
Figure \ref{fig:Guassian} plots the relative error against the iteration count for REBK and the proposed rRABEBK variants applied to Gaussian matrices of six different sizes.  Across all problem dimensions, the rRABEBK methods exhibit markedly faster convergence than REBK, confirming the performance gap already seen in Table \ref{table:Guassian}.

\subsection{Numerical Experiments on Structured Random Matrices} \label{subsubsub:rank-deficient}
This experiment is devoted to evaluating the performance of the proposed methods on structured random matrices, which exhibit nonuniform row norms and correlation patterns that more closely resemble those arising in practical large-scale applications. The structured random matrices are  generated following the procedure described in \cite{du2020randomized}. For completeness, we briefly summarize the construction procedure below.

Given $m$, $n$, a target rank $r = \textrm{rank}(A)$, and a condition number $\kappa>1$, 
the coefficient matrix $A$ is constructed as
\[
    A = U D V^\top,
\]
where $U \in \mathbb{R}^{m\times r}$ and $V \in \mathbb{R}^{n\times r}$ are orthonormal matrices obtained via
\[
    {\tt [U,\sim] = qr(randn(m,r),0),\qquad [V,\sim] = qr(randn(n,r),0)}
\]
in \texttt{MATLAB} and  $D \in \mathbb{R}^{r\times r}$ is diagonal with entries sampled uniformly from $(1,\kappa)$:
\[
    {\tt D = diag(1+(\kappa-1).*rand(r,1))}.
\]
This construction ensures that the condition number of $A$ does not exceed $\kappa$.

Table~\ref{table:rank_deficient} summarizes the numerical performance of REBK and the rRABEBK variants on the structured random matrices. All rRABEBK variants substantially outperform REBK, both in terms of iteration count and CPU time, highlighting the benefit of averaging over multiple hyperplanes and the acceleration afforded by appropriately chosen relaxation parameters.

Among the variants, rRABEBK-2.25 consistently achieves the best performance across all tested sizes, with the exception of the first example, achieving the lowest CPU times. For instance, for a matrix of size $2000 \times 4000$, rRABEBK-2.25 completes in only $101.691$ seconds, compared with $137.511$ seconds for rRABEBK-1.75, $216.630$ seconds for rRABEBK-1, and 
$2889.595$ seconds for REBK. This clearly demonstrates the effectiveness of both the
averaging scheme and a properly tuned relaxation factor in accelerating convergence.

\begin{table}[!htbp] 
\scriptsize 
\centering 
\caption{Numerical performance of REBK and rRABEBK variants on sparse least-squares problems with structured random matrices.}\label{table:rank_deficient}  
\resizebox{\textwidth}{!}{%
\begin{tabular}{|c|c|c|c|c|c|c|c|c|c|c|}
\hline
\multirow{2}{*}{$m\times n$} & \multirow{2}{*}{rank} & \multirow{2}{*}{$\kappa$} 
&\multicolumn{2}{c|}{REBK} & \multicolumn{2}{c|}{rRABEBK-1} & \multicolumn{2}{c|}{rRABEBK-1.75} & \multicolumn{2}{c|}{rRABEBK-2.25} \\ 
\cline{4-11} 
& & & IT & CPU & IT & CPU & IT & CPU & IT & CPU \\ 
\hline  
$1000\times 500$   & 480  & 10 & 92738 & 4.579   & 7351 & 1.583  & 4135 & 0.956 & 3107 & 1.207 \\  
$500\times 1000$   & 480  & 10 & 209428 & 11.286 & 16009 & 3.648 & 9035 & 2.824 & 7201 & 1.567 \\ 
$2000\times 1000$  & 900  & 5  & 84451 & 68.590  & 5378 & 7.369  & 3102 & 4.127 & 2548 & 3.113 \\  
$1000\times 2000$  & 900  & 5  & 130665 & 104.763 & 8658 & 12.197 & 5044 & 6.297 & 4074 & 4.691 \\   
$4000\times 2000$  & 1500 & 2  & 319759 & 1315.550 & 18151 & 83.782 & 11713 & 55.252 & 9977 & 47.970 \\   
$2000\times 4000$  & 1500 & 2  & 768752 & 2889.595 & 45431 & 216.630 & 28255 & 137.511 & 21502 & 101.691 \\   
\hline  
\end{tabular}
}
\end{table} 

To further illustrate the convergence behavior, Figure~\ref{fig:rank_deficient} 
plots the relative error versus iteration count for the same matrices. 
The figure clearly shows that all rRABEBK variants converge substantially faster than REBK,
with rRABEBK-2.25 consistently exhibiting the most rapid decay of the relative error.
These dynamic convergence curves confirm the tabulated results.

\begin{figure}[!htbp] 
\centering
	\subfigure[$m=1000,n=500,r=480,\kappa=10$]   
	{
		\begin{minipage}[t]{0.31\linewidth}
			\centering
			\includegraphics[width=1\textwidth]{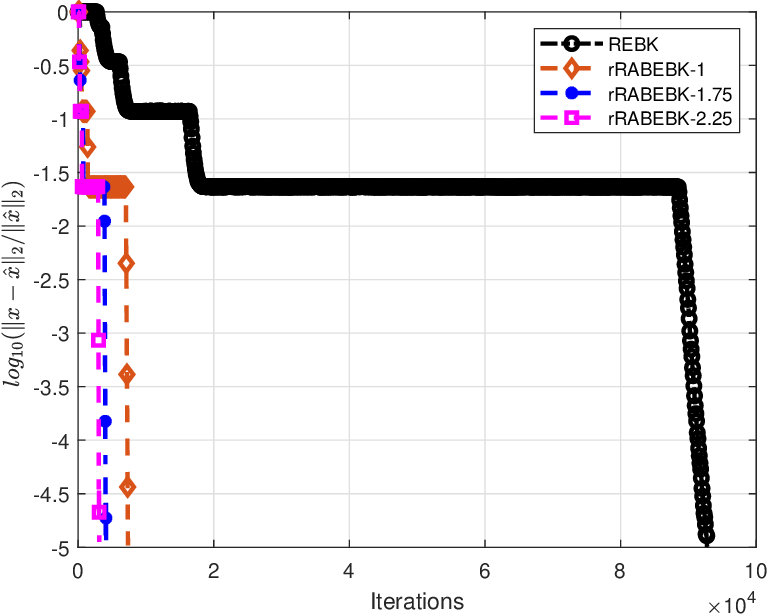}
		\end{minipage}
	}
	\subfigure[$m=2000,n=1000,r=900,\kappa=5$]  
	{
		\begin{minipage}[t]{0.31\linewidth}
			\centering
			\includegraphics[width=1\textwidth]{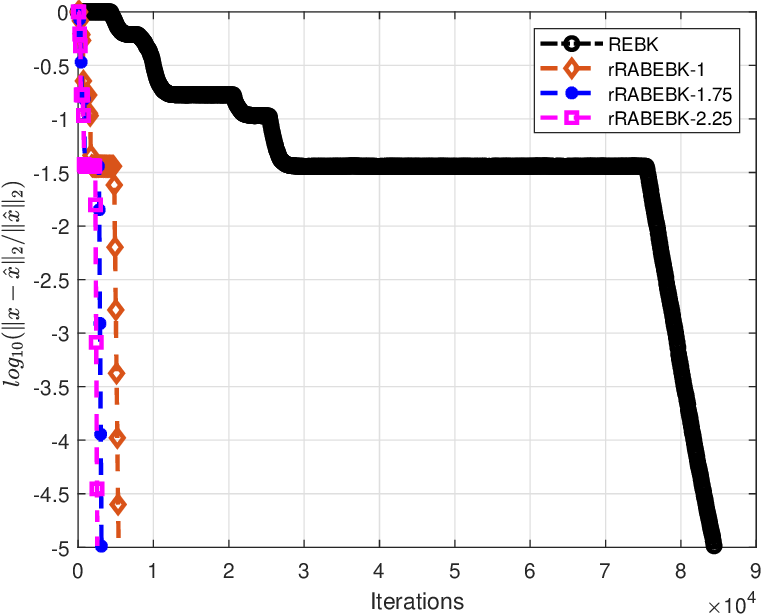}
		\end{minipage}
    }
    \subfigure[$m=4000,n=2000,r=1500,\kappa=2$]   
	{
		\begin{minipage}[t]{0.31\linewidth}
			\centering
			\includegraphics[width=1\textwidth]{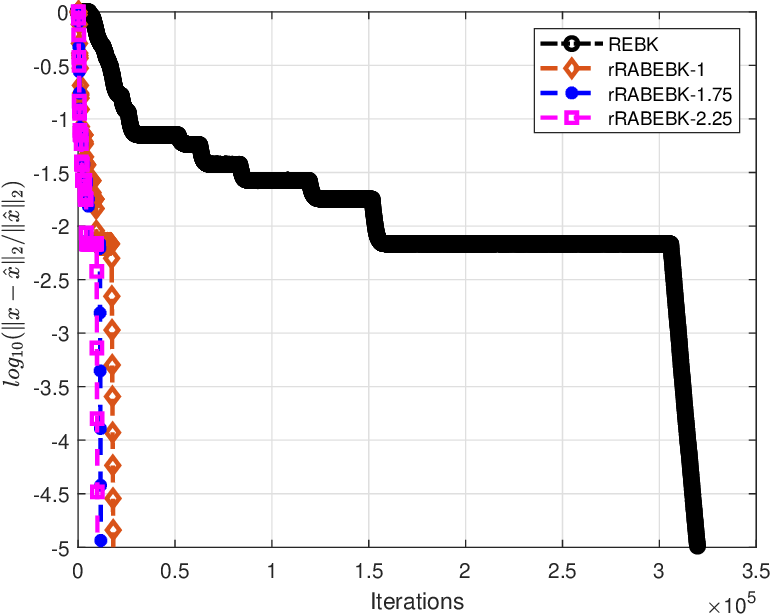}
		\end{minipage}
	}

	\subfigure[$m=500,n=1000,r=480,\kappa=10$]   
	{
		\begin{minipage}[t]{0.31\linewidth}
			\centering
			\includegraphics[width=1\textwidth]{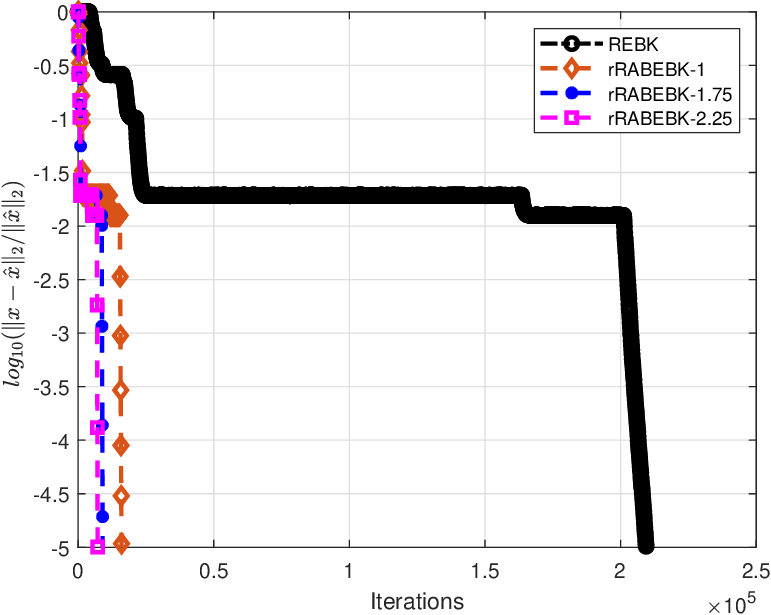}
		\end{minipage}
	}
	\subfigure[$m=1000,n=2000,r=900,\kappa=5$]  
	{
		\begin{minipage}[t]{0.31\linewidth}
			\centering
			\includegraphics[width=1\textwidth]{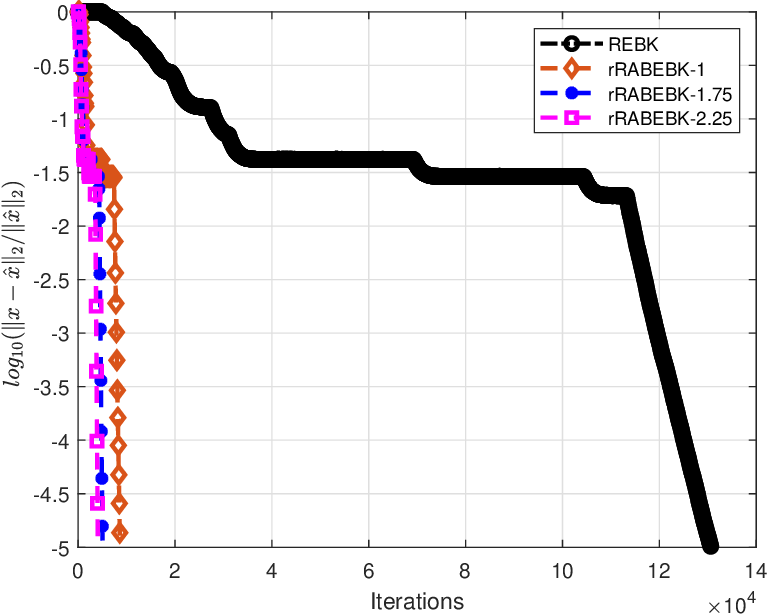}
		\end{minipage}
    }
    \subfigure[$m=2000,n=4000,r=1500,\kappa=2$]   
	{
		\begin{minipage}[t]{0.31\linewidth}
			\centering
			\includegraphics[width=1\textwidth]{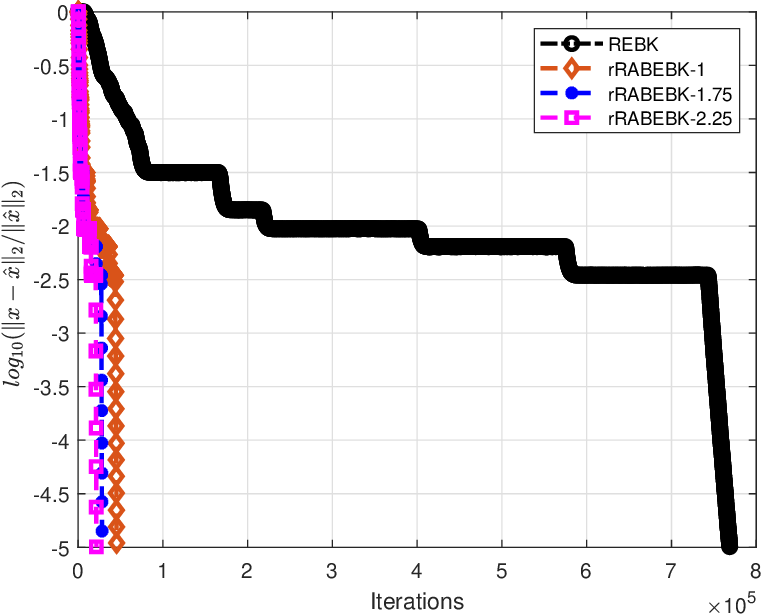}
		\end{minipage}
	}

\caption{The curves of relative error versus the number of iterations for overdetermined constructed matrices (top) and underdetermined constructed matrices (bottom).} 	
		\label{fig:rank_deficient}
\end{figure}

Overall, these experiments confirm that the proposed rRABEBK framework is robust and
highly efficient for structured random systems with controlled rank and condition number.

\subsection{Experiment on image recovery}
To further demonstrate the effectiveness of the proposed methods in practical scenarios, we consider an image recovery task based on the MNIST dataset, which is widely used in machine learning and computer vision. The image recovery of the MNIST dataset is available in the TensorFlow framework \cite{abadi2016tensorflow}. 

The coefficient matrix $A \in \mathbb{R}^{500 \times 784}$ is generated with i.i.d. Gaussian entries. The right-hand side vector is constructed as $$ b = A \hat{x} + e, $$ where $\hat{x} \in \mathbb{R}^{784}$ denotes the vectorized form of an MNIST image of size $28 \times 28$, and $e$ represents additive noise generated as in the previous numerical experiments.

The quality of the recovered image is evaluated using the peak signal-to-noise ratio (PSNR), defined as
\begin{equation}
\text{PSNR} := 10 \log_{10} \frac{ \sum_{i=1}^n x_i^2 }{ \sum_{i=1}^n (x_i - \hat{x}_i)^2 }.
\end{equation}

Figure \ref{fig:mnist_our_noise} presents the reconstruction results of REBK and the proposed rRABEBK variants after 10,000 iterations. It is evident that all rRABEBK variants achieve substantially higher-quality reconstructions than REBK. In particular, rRABEBK-2.25 attains the best reconstruction, with a PSNR of $33.095$, compared to $29.758$ for rRABEBK-1.75, $22.568$ for rRABEBK-1, and only $13.254$ for REBK.

These results confirm that the averaging block strategy, when combined with appropriately chosen relaxation parameters, substantially enhances the quality of image recovery. Moreover, the performance trends observed here are consistent with the convergence behaviour and efficiency improvements demonstrated in the previous synthetic experiments.

\begin{figure}[!htbp]
\centering
\begin{minipage}[c]{0.34\linewidth}
  \centering
  \subfigure[Exact phantom]{
      \includegraphics[width=0.82\linewidth]{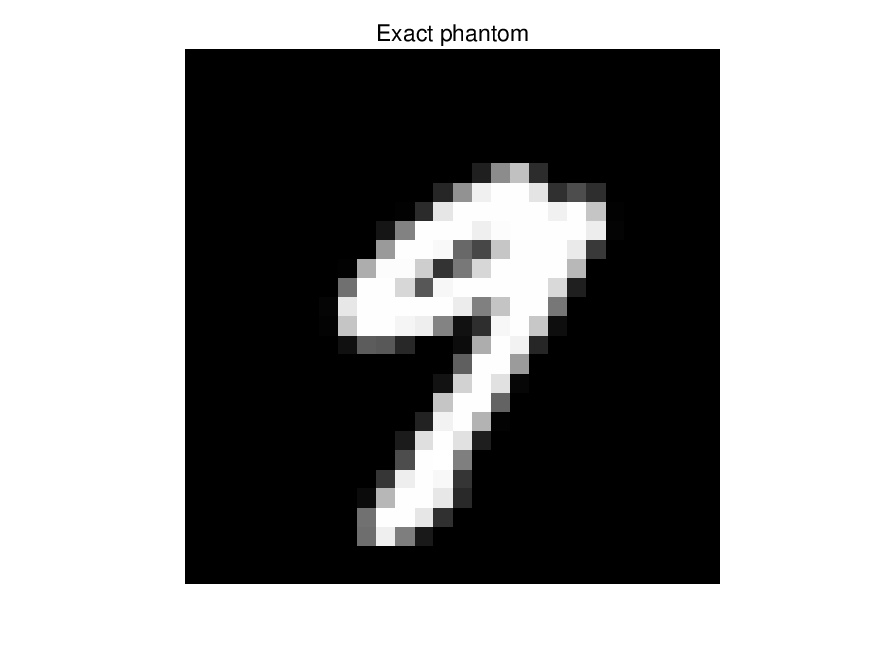}
  }
\end{minipage}%
\hfill
\begin{minipage}[c]{0.62\linewidth}
  \centering
  \subfigure[REBK, 13.254]{
      \includegraphics[width=0.45\linewidth]{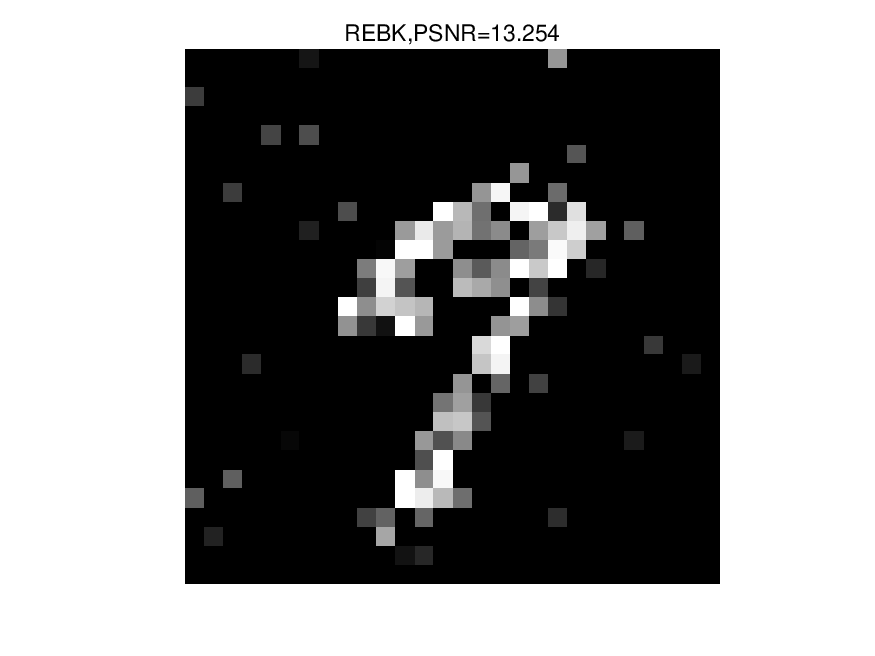}
  }
  \subfigure[rRABEBK-1, 22.568]{
      \includegraphics[width=0.45\linewidth]{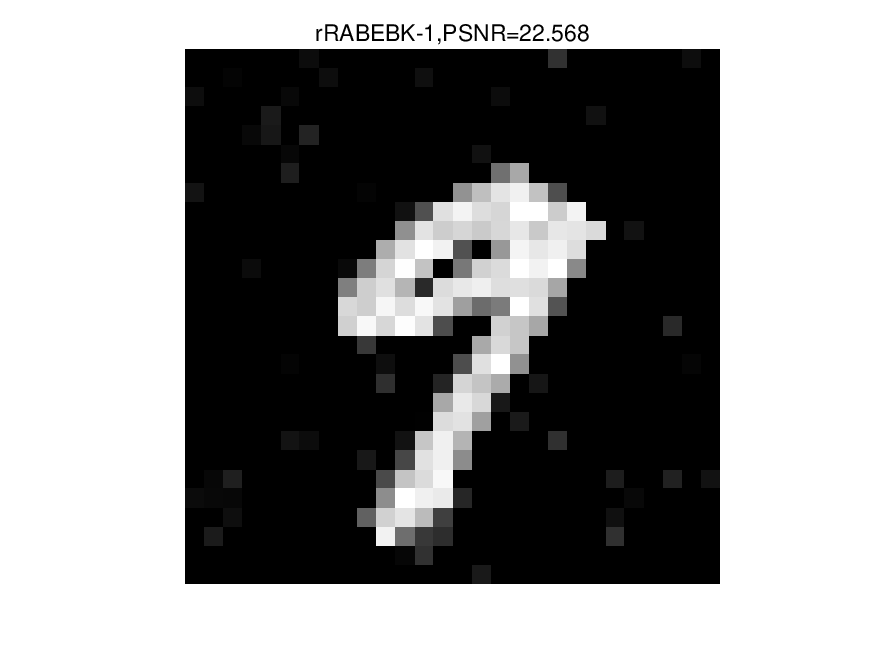}
  }\\[1ex]
  \subfigure[rRABEBK-1.75, 29.758]{
      \includegraphics[width=0.45\linewidth]{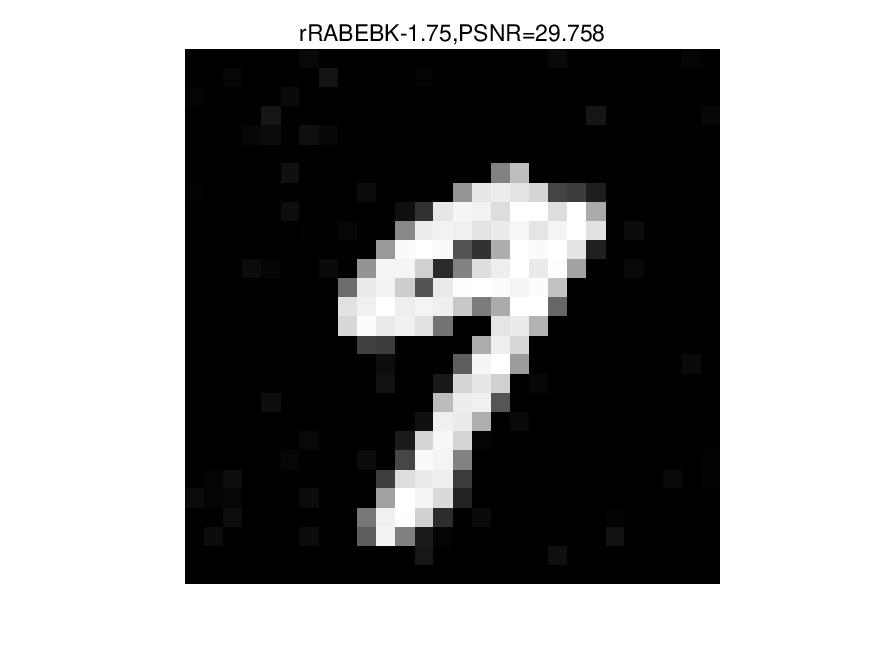}
  }
  \subfigure[rRABEBK-2.25, 33.095]{
      \includegraphics[width=0.45\linewidth]{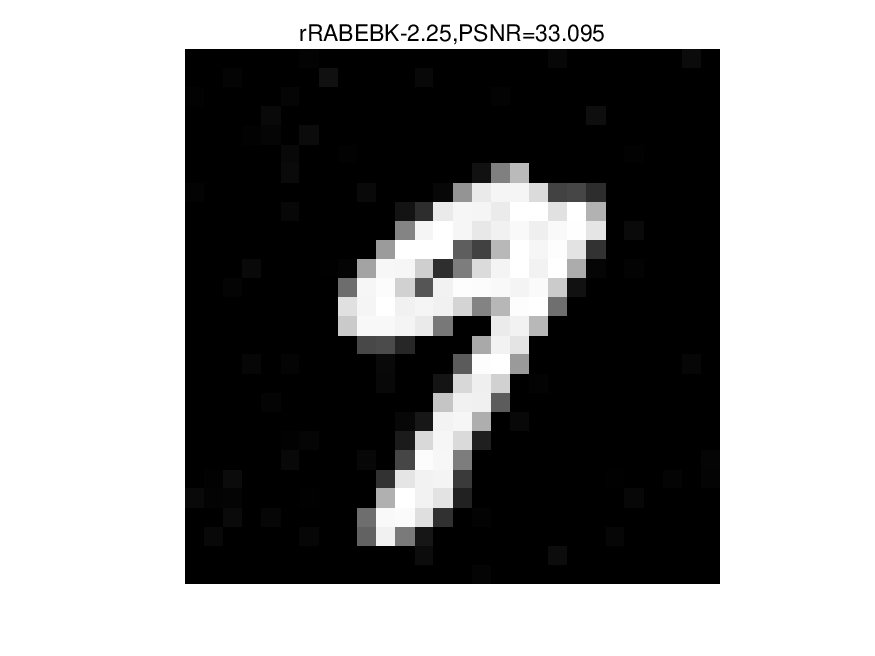}
  }
\end{minipage}
\caption{Original MNIST image (left) and reconstructed images obtained by REBK and rRABEBK variants after 10,000 iterations. The numbers below each reconstruction indicate the PSNR values (in dB).}
\label{fig:mnist_our_noise}
\end{figure}

\section{Conclusion} \label{sec:conclusion}
In this work, we have proposed a relaxed randomized averaging block extended Bregman-Kaczmarz (rRABEBK) framework for solving combined optimization problems. We have also investigated its convergence properties. By appropriately choosing the objective function $f$ and the data misfit function $g^*$, the proposed framework can be effectively applied to sparse least-squares problems.
Numerical experiments confirm the efficiency of the proposed methods, showing significant improvements over existing Kaczmarz-type algorithms. 

Future research directions include developing adaptive strategies for selecting the relaxation parameter, as well as explore broader choices of the functions $f$ and $g^*$, with the aim of extending the applicability of the framework to a wider range of inverse problems.

\bibliographystyle{siamplain}
\bibliography{ref_RABEBK}
\end{document}